\documentclass{amsart}
\usepackage{amsmath,latexsym,amsfonts,amssymb,amsthm,amscd,graphics,
graphicx,hyperref}
\usepackage[dvipsnames]{xcolor}

\newtheorem{thm}{Theorem}[section]
\newtheorem{lemma}[thm]{Lemma}
\newtheorem{prop}[thm]{Proposition}
\newtheorem{cor}[thm]{Corollary}

\theoremstyle{definition}

\newtheorem{rem}[thm]{Remark}

\numberwithin{equation}{section}

\author{Serban T. Belinschi}
\address{Institut de Math\'ematiques de Toulouse; UMR5219; Universit\'e de Toulouse; CNRS; UPS, F-31062 Toulouse,
FRANCE}
\email{serban.belinschi@math.univ-toulouse.fr}
\author{Hari Bercovici}
\address{Department of Mathematics and Statistics, 
Indiana University, Bloomington, IN 47405 U.S.A.}
\email{bercovic@indiana.edu}
\author{Mireille Capitaine}
\address{Institut de Math\'ematiques de Toulouse; UMR5219; Universit\'e de Toulouse; CNRS; UPS, F-31062 Toulouse,
FRANCE}
\email{mireille.capitaine@math.univ-toulouse.fr}

\thanks{HB was supported by a grant from the National Science Foundation. This work was started
while HB was visiting the Institute of Mathematics of Toulouse as Professeur Invit\'e.}

\title[Outlying eigenvalues of a polynomial in large random matrices]{On the 
outlying eigenvalues of a polynomial in large independent random matrices}

\begin{document}


\begin{abstract}
Given a selfadjoint polynomial $P(X,Y)$ in two noncommuting selfadjoint indeterminates, we 
investigate the asymptotic eigenvalue behavior 
of the random matrix $P(A_N,B_N)$, where 
$A_N$ and $B_N$ are independent Hermitian random 
matrices and the distribution of $B_N$ is invariant 
under conjugation by unitary operators. 
We assume that the empirical eigenvalue 
distributions of $A_N$ and $B_N$ converge 
almost surely to deterministic probability 
measures $\mu $ and $\nu$, respectively. 
In addition, the eigenvalues of $A_N$ and 
$B_N$ are assumed to converge uniformly 
almost surely to the support of $\mu$ and 
$\nu,$ respectively, except for a fixed finite 
number of fixed eigenvalues (spikes) of $A_N$.
It is known that almost surely the empirical distribution of the 
eigenvalues of $P(A_N,B_N)$ converges to a certain
deterministic probability measure $\eta$ (sometimes denoted $\eta=P^\square(\mu,\nu)$) 
 and, when there
are no spikes, the eigenvalues of $P(A_N,B_N)$ converge
uniformly almost surely to the support of $\eta$. When
spikes are present, we show that the eigenvalues of 
$P(A_N,B_N)$ still converge uniformly to 
the support of $\eta$, with the possible exception
of certain isolated outliers whose location can be 
determined in terms of $\mu,\nu,P$, and the 
spikes of $A_N$. We establish a similar result when 
$B_N$ is replaced by a Wigner matrix. The relation between outliers 
and spikes is described using the operator-valued  subordination 
functions of free probability theory. These results 
extend known facts from the special case in which 
$P(X,Y)=X+Y$.
\end{abstract}

\maketitle

\section{Introduction}
Let $\mu$ and $\nu$ be two Borel probability measures with bounded support on $\mathbb R$.  Suppose given, for each positive integer $N$, 
selfadjoint $N\times N$ independent random matrices 
$A_N$ and $B_N$, with the following properties:
\begin{enumerate}

\item[(a)] the distribution of  $B_N$ is invariant 
under conjugation by unitary $N\times N$
matrices;

\item[(b)] the empirical 
eigenvalue distributions of $A_N$ and $B_N$ converge 
almost surely to $\mu$ and $\nu$, respectively;

\item[(c)] the eigenvalues of $A_N$ and $B_N$ converge 
uniformly almost surely to the supports of $\mu$
and $\nu$, respectively, with the exception of a fixed number $p$ of \emph{spikes}, that is, 
fixed eigenvalues of $A_N$ that lie outside the support of $\mu$.

\end{enumerate}

{When spikes are absent, that is, when $p=0$,}
 it was shown in \cite{CM} that
the eigenvalues of 
$A_N+B_N$ converge uniformly almost 
surely to the support of the free additive 
convolution $\mu\boxplus\nu$. When $p>0$,
the eigenvalues of $A_N+B_N$ also converge 
uniformly almost surely to a compact set 
$K\subset\mathbb R$ such that 
{$\mathrm{supp}(\mu\boxplus\nu)\subset K$ and}
$K\setminus\mathrm{supp}(\mu\boxplus\nu)$ 
has no accumulation points in $\mathbb R
\setminus\mathrm{supp}(\mu\boxplus\nu)$. 
Moreover, if $t\in K\setminus\mathrm{supp}
(\mu\boxplus\nu)$, then $\omega(t)$ 
is one of the spikes of $A_N$, where $\omega$ 
is a certain subordination function
arising in free probability. 
The relative position of the eigenvectors corresponding to  
spikes and outliers is also given in terms of subordination functions.
We refer to \cite{BBCF} for this result.

Our  purpose is to show that 
analogous results hold when the sum 
$A_N+B_N$ is replaced by an arbitrary 
selfadjoint polynomial $P(A_N,B_N)$.
Then, by a comparison procedure to the particular
case when $B_N$ is a G.U.E. (Gaussian unitary
ensemble),  we are  also able to identify the outliers of 
an arbitrary selfadjoint polynomial $P(A_N,\frac{X_N}{\sqrt{N}})$ 
when $X_N$ is a Wigner matrix independent from $A_N$. This extends an earlier 
result \cite{CDMFF} pertaining to
 additive deformations 
of Wigner matrices.
More precisely we consider a Hermitian matrix
 $X_{N}=[X_{ij}]_{i,j=1}^N$, where $[X_{ij}]_{i\geq1,j\geq1}$ is 
an infinite array of random variables such that
 \begin{enumerate}\item[(X0)] $X_N$ is independent from $A_N$,
\item[(X1)] $X_{ii}$, $\sqrt{2}\Re(X_{ij}), i<j$, 
$\sqrt{2}\Im(X_{ij}), i<j$, are independent, centered with variance 1,
 \item[(X2)] there exist 
 $K,x_0>0$, $n_0\in\mathbb N$, and a random variable $Z$ 
with finite fourth moment
such that
\begin{equation}\frac{1}{n^2} \sum_{1\leq i,j\leq n}\mathbb{P}\left(\vert X_{ij}\vert >x\right) \leq 
K\mathbb{P}\left(\vert Z \vert>x\right)\quad x>x_0,n>n_0.\nonumber \end{equation}
\item[(X3)] $\sup\{\mathbb{E}(\vert X_{ij}\vert^3):i,j\in\mathbb N,i<j\}<+\infty.$ 
\end{enumerate}

\begin{rem} 
The matrix $X_N$ is called a G.U.E. if the variables $X_{ii}$, $\sqrt{2}\Re(X_{ij}),{i<j}$, and $\sqrt{2}\Im(X_{ij}),{i<j}$, are independent standard Gaussian.
Assumptions (X2) and (X3) obviously hold if 
 these variables are merely independent and
identically distributed with a finite fourth moment. 
\end{rem}

Our result lies in the lineage of recent, and not so recent, works 
\cite{BaiYao08b,BBP05,BaikSil06,BGR,C,Capitaine14,CDF09,CDMFF,FePe,FurKom81, John,LV,My,
Peche06,PRS,RaoSil09} studying the influence
of additive or multiplicative perturbations on the extremal
eigenvalues of classical random matrix models, the seminal paper being
\cite{BBP05}, where the so-called BBP phase transition was observed.

We note that Shlyakhtenko \cite{ShlB} 
considered a framework which makes it possible to
understand this kind of result as a manifestation of 
infinitesimal freeness. In fact, the results of
\cite{ShlB} also allow one to detect the presence of 
spikes from the behaviour of the bulk of 
the eigenvalues of $P(A_N,B_N)$, even when 
$P(A_N,B_N)$ has no outlying eigenvalues. In a related 
result, Collins, Hasebe and Sakuma \cite{CHS}
study the 
{`purely spike'}
 case in which $\mu=\nu=\delta_0$ and the 
eigenvalues of $A_N$ and $B_N$ accumulate
to given sequences $(a_k)_{k=1}^\infty$ 
and $(b_k)_{k=1}^\infty$ of real numbers
converging to zero.

\section{Notation and preliminaries on strong asymptotic freeness}

We recall that a $C^*$-{\em probability space} is a 
pair $(\mathcal A,\tau)$, where $\mathcal A$ is
a $C^*$-algebra and $\tau$ is a state on $\mathcal A$. 
We always assume that $\tau$ is
faithful.  The 
elements of $\mathcal A$ are 
referred to as {\em random variables}.

If $(\Omega,\Sigma,P)$ is a classical probability space,
then $(L^\infty (\Omega),\mathbb{E})$ 
is a 
$C^*$-probability space, 
where $\mathbb E$ is the usual expected value. 
Given $N\in\mathbb N$,
$(M_N (\mathbb C),{\rm tr}_N)$ is a $C^*$-probability space, 
where ${ \rm tr}_N=\frac1N{\rm Tr}_N$ 
denotes the normalized trace.  
More generally, if $(\mathcal A,\tau)$ is 
an arbitrary $C^*$-probability space and 
$N\in\mathbb N$, 
then $M_N(\mathcal A)=M_N(\mathbb C)\otimes \mathcal A$ 
becomes a $C^*$-probability 
space with the state ${\rm tr}_N\otimes\tau$.

The {\em distribution} $\mu_a$ of a selfadjoint element $a$ in 
a $C^*$-{probability space} $(\mathcal A,\tau)$ is a 
compactly supported probability measure on $\mathbb R$,
uniquely determined by the requirement that 
$\int_\mathbb Rt^n\,d\mu_a(t)=\tau(a^n)$, $n\in\mathbb N$. 
The {\em spectrum} of an element $a\in\mathcal A$ is
$$
\sigma(a)=\{\lambda\in\mathbb C\colon\lambda1-a\text{ is not invertible in  }\mathcal A\}.
$$
For instance, if $A\in M_N(\mathbb C)$ is a selfadjoint matrix, 
then the distribution of $A$ relative to ${\rm tr}_N$ 
is the measure $\mu_A=\frac1N\sum_{j=1}^N\delta_{\lambda_j(A)}$, 
where $\{\lambda_1(A),\dots,
\lambda_N(A)\}$ is the list of the 
eigenvalues of $A$, repeated according 
to multiplicity. As usual, the support $\text{supp}(\mu)$
of a Borel probability measure $\mu$ on $\mathbb R$ is the smallest 
closed set $F\subset\mathbb R$ with the property that $\mu(F)=1$.
It is known that if $a=a^*\in\mathcal A$
{and $\tau$ is faithful,}
 then 
$
\sigma(a)=\mathrm{supp}(\mu_a).
$
In the following, we assume that $\tau$ is a tracial state, that is, $\tau(ab)=\tau(ba),a,b\in\mathcal A$.

Suppose that we are given $C^*$-{probability spaces} 
$\{(\mathcal A_N,\tau_N)\}_{N=0}^\infty$ and 
selfadjoint elements $a_N\in\mathcal A_N$, $N\ge0$. We say 
that $\{a_N\}_{N=1}^\infty$ {\em converges
in distribution} to $a_0$ if 
\begin{equation}\label{unu}
\lim_{N\to\infty}\tau_N(a_N^k)=\tau_0(a_0^k),\quad k\in\mathbb N.
\end{equation}
We say that $\{a_N\}_{N=1}^\infty$ {\em converges strongly in distribution} 
to $a_0$ (or to $\mu_{a_0}$) 
if, in addition to \eqref{unu}, the sequence 
$\{\mathrm{supp}(\mu_{a_N})\}_{N=1}^\infty$
converges to $\mathrm{supp}(\mu_{a_0})$ in the Hausdorff metric. 
This condition simply means that
 for every $\varepsilon>0$ 
there exists $N(\varepsilon)\in
\mathbb N$ such that 
$$
\mathrm{supp}(\mu_{a_N})\subset\mathrm{supp}(\mu_{a_0})+
(-\varepsilon,\varepsilon)
$$
and
$$
\mathrm{supp}(\mu_{a_0})\subset\mathrm{supp}(\mu_{a_N})+
(-\varepsilon,\varepsilon)
$$
for every $ N\ge N(\varepsilon)$.
If all the traces $\tau_N$ are faithful, 
{strong convergence}
can be reformulated as follows:
$$
\lim_{N\to\infty}\|P(a_N)\|=\|P(a_0)\|,
$$
for every polynomial $P$ with complex coefficients. 
This observation allows us to extend the concept
of (strong) convergence in distribution to $k$-tuples of 
random variables, $k\in\mathbb N$. For every
$k\in\mathbb N$, we denote by $\mathbb C\langle X_1,\dots,X_k\rangle$ 
the algebra of polynomials 
with complex coefficients in $k$ noncommuting indeterminates 
$X_1,\dots,X_k$. This is a 
$*$-algebra with the adjoint operation determined by
$$
(\alpha X_{i_1}X_{i_2}\cdots X_{i_n})^*=
\overline{\alpha}X_{i_n}\cdots X_{i_2}X_{i_1},\quad
\alpha\in\mathbb C,\ i_1,i_1,\dots,i_n\in\{1,\dots,k\}.
$$

Suppose that $\{(\mathcal A_N,\tau_N)\}_{N=0}^\infty$ is a 
sequence of $C^*$-probability spaces,
$k\in\mathbb N$, and $\{a_N\}_{N=0}^\infty$ is a sequence 
of $k$-tuples $a_N=(a_{N,1},\dots,
a_{N,k})\in\mathcal A_N^k$ of selfadjoint elements. 
We say that $\{a_N\}_{N=1}^\infty$ {\em converges
in distribution} to $a_0$ if 
\begin{equation}\label{due}
\lim_{N\to\infty}\tau_N(P(a_N))=\tau_0(P(a_0)),
\quad P\in\mathbb C\langle X_1,\dots,X_k\rangle.
\end{equation}
We say that $\{a_N\}_{N=1}^\infty$ {\em converges strongly
in distribution} to $a_0$ if, in addition to \eqref{due}, we have 
$$
\lim_{N\to\infty}\|P(a_N)\|=\|P(a_0)\|,
\quad P\in\mathbb C\langle X_1,\dots,X_k\rangle.
$$

The above concepts extend to $k$-tuples $a_N=(a_{N,1},\dots,
a_{N,k})\in\mathcal A_N^k$ which do not necessarily consist of 
selfadjoint elements. The only change is 
that one must use polynomials in the variables $a_{N,j}$ and 
their adjoints $a_{N,j}^*$, $j=1,\dots,k$.

\begin{rem}\label{rem2.1}
Suppose that all the states $\tau_N,N\in\mathbb N$, are faithful. 
{As seen}
in \cite[Proposition 2.1]{CM}, $\{a_N\}_{N=1}^\infty$ converges strongly in 
distribution to $a_0$ if and only if
$\{P(a_N)\}_{N=1}^\infty$ converges strongly in 
distribution to $P(a_0)$ for every selfadjoint polynomial
$P\in\mathbb C\langle X_1,\dots,X_k\rangle.$ 
Moreover, strong convergence 
in distribution also implies strong convergence at the matricial level. 
The following result is 
\cite[Proposition 7.3]{M}.
\end{rem}

\begin{prop}\label{Camille}
Let $\{(\mathcal A_N,\tau_N)\}_{N=0}^\infty$ be $C^*$-probability 
spaces with faithful states
$\{\tau_N\}_{N=0}^\infty$, let $k\in\mathbb N$, and let $\{a_N\}_{N=0}^\infty$ be a sequence of 
$k$-tuples of selfadjoint elements $a_N\in\mathcal A_N^k$. 
Suppose that $\{a_N\}_{N=1}^\infty$
converges strongly in distribution to $a_0$. Then $\lim_{N\to\infty}\|P(a_N)\|=\|P(a_0)\|$ 
for every $n\in\mathbb N$ and every matrix polynomial
$P\in M_n(\mathbb C\langle X_1,\dots,X_k\rangle)$.
\end{prop}

A special case of strong convergence in distribution arises from 
the consideration of random matrices in $M_N(\mathbb C)$. 
The following result follows from \cite[Theorem 1.4]{CM} and \cite[Theorem 1.2]{BC}.

\begin{thm}\label{stronguni}
Let $(\mathcal A_N,\tau_N)$ denote the space 
$(M_N(\mathbb C),\rm tr_N)$,
$N\in\mathbb N$. Suppose that $k_1,k_2, k_3\in\mathbb N$ are fixed,
$u_N=(U_{N,1},\dots,U_{N,k_1})$,  $x_N=(X_{N,1},\dots,X_{N,k_2})$ and 
$a_N=(A_{N,1},\dots,A_{N,k_3})$ are mutually independent random 
tuples 
{of matrices} in some classical probability space such that:
\begin{itemize}
\item[(i)]  $U_{N,1},\dots,
U_{N,k_1}$ are independent unitaries distributed according to 
the Haar measure on the unitary group ${\rm U}(N),N\in
\mathbb N$.
\item[(ii)]  $X_{N,1},\dots,
X_{N,k_2}$ are independent Hermitian matrices, each 
satisfying assumptions $(X1), (X2),$ and $(X3)$
 in the introduction.
\item[(iii)]  $a_N$ is a vector of $N\times N$ selfadjoint matrices 
such that the sequence $\{a_N\}_{N=1}^\infty$ converges 
strongly almost surely in distribution to some deterministic $k_3$-tuple
in a $C^*$-probability space.
\end{itemize}
Then there exist a $C^*$-probability space $(\mathcal A,\tau)$, a free family  
$u=(u_1,\dots,u_{k_1})\in\mathcal A^{k_1}$ of Haar unitaries,  
a semicircular system $ x=(x_1,\dots,x_{k_2})\in\mathcal A^{k_2}$ and 
$a=(a_1,\dots,a_{k_3})\in\mathcal A^{k_3}$, such that, $u,x,$ and $a$ are free and 
$\{(u_N,x_N,a_N)\}_{N=1}^\infty$  converges strongly almost surely in distribution to $(u, x, a)$.
\end{thm}

We recall that  a tuple $(x_1,\dots,x_{k})$ of elements in a ${C}^*$-probability 
space  $\left({\mathcal A}, \tau\right)$ is
{called}
a {\em semicircular system} if
$\{x_1,\dots,x_{k}\}$ is a free family of selfadjoint random variables,
and for every $i=1,\ldots,k$,   $\mu_{x_i}$ 
is the standard semicircular
distribution $\nu_{0,1}$ defined by
\begin{equation}\label{demicercle}
{\rm d}\nu_{0,1}(t)=
\frac{1}{2\pi} \sqrt{4-t^2}{1\!\!{\sf I}}_{[-2,2]}(t)\,\mathrm{d}t.
\end{equation}  
An element $u\in\mathcal A$ is
called a {\em Haar unitary} if $u^*=u^{-1}$ and 
$\tau(u^n)=0$ for all $n\in\mathbb Z\setminus\{0\}$.
Note that Theorem 1.2 in \cite{BC} deals with deterministic  
$a_N$ but the random case readily follows as pointed out by assertion 2 in \cite[Section 3]{M}. 
The point of Theorem \ref{stronguni} is, of course, that the resulting convergence is strong.
Convergence in distribution was established earlier  
(see \cite{V-Invent}, \cite{D}, \cite[Theorem 5.4.5]{AGZ}).

We also need a simple coupling result from \cite[Lemma 5.1]{CM}.
\begin{lemma}\label{fortdiag}
Suppose given selfadjoint  matrices $C_N,D_N\in M_N(\mathbb C)$, $N\in\mathbb N$, 
such that the sequences $\{C_N\}_{N\in\mathbb N}$ and $\{D_N\}_{N\in\mathbb N}$ 
converge strongly in distribution.
Then there exist diagonal matrices 
$\widetilde{C}_N,\widetilde{D}_N\in M_N(\mathbb C)$, $N\ge1$, 
such that $\mu_{\widetilde{C}_N}=\mu_{C_N}$, $\mu_{\widetilde{D}_N}=\mu_{D_N}$, 
 and the sequence 
$\{(\widetilde{C}_N,\widetilde{D}_N)\}_{N\in\mathbb N}$ converges 
strongly in distribution.
\end{lemma}

\section{Description of the models}\label{model}

In order to describe in 
detail our  matrix models, we need two compactly supported
probability measures $\mu$ and $\nu$ on $\mathbb R$, a positive 
integer $p$, and a sequence of fixed real numbers 
$\theta_1\ge\theta_2\ge\cdots\ge\theta_p$ in $\mathbb R\setminus\text{supp}(\mu)$. 
The matrix $A_N\in M_N(\mathbb C)$ is random selfadjoint
 for all $N\in\mathbb N,
N\ge1$ and satisfies the following conditions:
\begin{itemize}
\item[(A1)] almost surely,  the sequence  $\{A_N\}_{N=1}^\infty$ converges in distribution  to $\mu$,
\item[(A2)] $\theta_1\ge\theta_2\ge\cdots\ge\theta_p$ are $p$ eigenvalues of $A_N$, and
\item[(A3)] the other eigenvalues of $A_N$, which may be random, 
converge uniformly almost surely to ${\rm supp}(\mu)$:
almost surely, for every $\varepsilon>0$ there exists $N(\varepsilon)\in\mathbb N$ 
such that
$$
\sigma(A_N)\setminus\{\theta_1,\dots,\theta_p\}\subseteq\mathrm{supp}(\mu)+(-\varepsilon, 
\varepsilon),\quad N\ge N(\varepsilon).
$$ 
In other words, only the  $p$ eigenvalues $\theta_1,\ldots,\theta_p$ prevent 
$\{A_N\}_{N=1}^\infty$ from converging strongly
in distribution to $\mu$.
\end{itemize}
We 
investigate two polynomial matricial models, both involving $A_N$. The first model  involves
a sequence $\{B_N\}_{N=1}^\infty$ of random Hermitian matrices such that 
\begin{itemize}
\item[(B0)] $B_N$ is independent from $A_N$,	
\item[(B1)] 
$B_N$  converges strongly in 
distribution to the compactly supported
probability measure $\nu$ on $\mathbb R$, 
\item[(B2)] for each $N$, the distribution of $B_N$ is invariant under conjugation 
by arbitrary $N\times N$ unitary matrices. 
\end{itemize}
The matricial model is
\begin{equation}\label{uni}
Z_N=P(A_N,B_N)
\end{equation} 
for an arbitrary selfadjoint polynomial $P\in\mathbb C\langle X_1,
X_2\rangle$.

The second model deals with 
 $N\times N$ random Hermitian Wigner matrices $X_N=[{X}_{ij}]_{i,j=1}^N$, where
$[X_{ij}]_{i\geq1,j\geq 1}$ is an infinite array of
random variables satifying conditions $(X0)-(X3)$ in the introduction.
The matricial model is
\begin{equation}\label{wig}
Z_N=P\left(A_N,\frac{X_N}{\sqrt{N}}\right)
\end{equation} 
for an arbitrary selfadjoint polynomial $P\in\mathbb C\langle X_1,X_2\rangle$.

In the discussion of the first model, we use results 
of Voiculescu \cite{V-Invent} (see also \cite{VDN}), 
who showed that 
there exist a free pair $(a,b)$ of 
selfadjoint elements 
 in a II${}_1$-factor $(\mathcal A,\tau)$ such that, almost surely,  the 
sequence $\{(A_N,B_N)\}_{N=1
}^\infty$ converges in distribution to $(a,b)$. Thus,
$\mu=\mu_a,\nu=\mu_b$, and  the sequence $\{P(A_N,B_N)\}_{N=1}^\infty$ converges 
in distribution 
to $P(a,b)$ (that is, $$
\lim_{N\to\infty}\mu_{P(A_N,B_N)}=\mu_{P(a,b)}
$$ in the weak${}^*$ topology) for every
 selfadjoint polynomial  
$P\in \mathbb C\langle X_1,
X_2\rangle$.  When $p=0$, Lemma \ref{fortdiag}, Theorem \ref{stronguni} 
and Remark \ref{rem2.1}, show that, almost surely, this convergence is strong (see the proof of Corollary 2.2 in \cite{CM}).

For the second model we use  \cite[Proposition 2.2]{BC} and \cite[Theorem 5.4.5]{AGZ}, where it is seen that for every selfadjoint polynomial 
$P\in \mathbb C\langle X_1,X_2\rangle$ we have
$$
\lim_{N\to\infty}\mu_{ P(A_N, {X_N}/{\sqrt{N}})}=\mu_{P(a,b)}
$$
almost surely in the weak${}^*$ topology, where
$a$ and $b$ are freely independent selfadjoint noncommutative random variables,
$\mu_a=\mu$, and $\mu_b=\nu_{0,1}$. 
As in the first model,  Theorem \ref{stronguni} 
and Remark \ref{rem2.1} show that, almost surely, 
the sequence $\{ P(A_N, {X_N}/{\sqrt{N}})\}_{N=1}^\infty$ converges 
strongly in distribution
to $P(a,x)$ provided that $p=0$.

Our main result applies, of course, to the case in which $p>0$. Let $Y_N$ be either 
$B_N$ or ${X_N}/{\sqrt{N}}$. 
The set of outliers of $P(A_N,Y_N)$ is calculated from the spikes 
$\theta_1,\dots,\theta_p$ using Voiculescu's
matrix subordination function \cite{V2000}. 
When $Y_N=B_N$, we also show that the eigenvectors 
associated to these outlying
eigenvalues have projections of computable size onto the 
eigenspaces of $A_N$ corresponding to the spikes. The precise statements are
 Theorems \ref{Main} and  \ref{thm8.2}.
 Sections \ref{sec:linearization} and
 \ref{sec:subordination} contain the necessary tools from 
operator-valued noncommutative probability theory while 
Sections \ref{sp}--\ref{cvFw} are dedicated to the proofs of the main results.

\section{Linearization}\label{sec:linearization}
As in \cite{Anderson,BMS}, we use   linearization to reduce a problem about
a polynomial in freely independent, or asymptotically freely 
independent, random variables, to a problem about the 
addition of matrices having these random variables as entries. Suppose that
 $P\in\mathbb C\langle X_1,\dots,X_k\rangle$. 
For our purposes, a \emph{linearization} of $P$ is a linear 
polynomial of the form
$$
z\alpha\otimes 1 -L,
$$
where $z$ is a complex variable, and
$$L=\gamma_0\otimes1+\gamma_1\otimes X_1
+\cdots+\gamma_k\otimes X_k,$$
with $\alpha,\gamma_0,\dots,\gamma_k\in M_n(\mathbb C)$
for some $n\in\mathbb N$, and the following property is satisfied: 
\emph{given $z\in\mathbb C$ and elements $a_1,\dots,a_k$ in 
a $C^*$-algebra $\mathcal A$, $z-P(a_1,\dots,a_k)$ is invertible 
in $\mathcal A$ if and only if $z\alpha\otimes 1-L(a_1,\dots,a_k)$ is 
invertible in $M_n(\mathcal A)$}. Usually, this is achieved by 
ensuring that $(z\alpha\otimes 1-L)^{-1}$ exists as an element of 
$M_n(\mathbb C\langle X_1,\dots,X_k\rangle\langle(z-P)^{-1}\rangle)$ and $(z-P)^{-1}$ 
is one of the entries of the $(z\alpha\otimes 1-L)^{-1}$.
It is known (see, for instance, \cite{MPS}) that every polynomial has
a linearization.
{See \cite{HT} for earlier uses of linearization in free probability.} 

In the following 
we also say, more concisely, that $L$ is a linearization of $P$.  
We also suppress the unit of the algebra 
$\mathcal A$ when there is no risk of confusion.  
For instance, we may write $z\alpha -L$ in place of
$z\alpha\otimes 1 -L.$

We describe in some detail a linearization procedure from \cite{Anderson} (see also \cite{Mai}) 
that has several advantages. In this procedure, we always have $\alpha=e_{1,1}$, where
 $e_{1,1}$ denotes the matrix whose only nonzero entry equals $1$
and occurs in the first row and first column.
Given  $P\in\mathbb C\langle X_1,\dots,X_k\rangle$, 
we produce  an integer $n\in\mathbb N$ and a linear polynomial 
$L\in M_n(\mathbb C\langle X_1,\dots,X_k\rangle)$ of the form
$$
L=\begin{bmatrix}
0 & u\\
v & Q
\end{bmatrix},
$$
such that $u\in M_{1\times (n-1)}(\mathbb C\langle X_1,\dots,X_k\rangle)$, 
$v\in M_{(n-1)\times 1}(\mathbb C\langle X_1\dots,X_k\rangle)$, 
 $Q$ is an invertible matrix in 
 $M_{n-1}(\mathbb C\langle X_1,\dots X_k\rangle)$ whose 
 inverse is a polynomial of degree less than or equal to the degree of $P$,
 and $uQ^{-1}v=-P$.
Moreover, if $P=P^*$, the coefficients of $L$ can be chosen to be 
selfadjoint matrices in $M_n(\mathbb C)$.

The construction proceeds by induction on the number of monomials in the given polynomial.  If $P$ is a monomial of degree $0$ or $1$, we set $n=1$ and $L=P$.
If $P=X_{i_1}X_{i_2}X_{i_3}\cdots X_{i_{\ell-1}}X_{i_\ell}$, where
$\ell\ge2$ and $i_1,\dots,i_\ell\in\{1,\dots,k\}$, we set $n=\ell$ and
$$
L=-\begin{bmatrix}
0 & 0 & \cdots &  0 & X_{i_1}\\
0 & 0 & \cdots &  X_{i_2} & -1\\
\vdots&\vdots& \reflectbox{$\ddots$}&\vdots&\vdots\\
0 & X_{i_{\ell-1}}&\cdots&0&0\\
X_{i_\ell}&-1&\cdots&0&0
\end{bmatrix}.
$$
As noted in  \cite{Mai},
the lower right $(\ell-1)\times(\ell-1)$ corner of this matrix has an inverse of degree $\ell-2$ in
 the algebra
$M_{\ell-1}(\mathbb C\langle X_1,\dots,X_k\rangle)$. 
(The constant term in this inverse is a selfadjoint matrix and 
its spectrum 
is contained in
$\{-1,1\}$.) Suppose now that $p=P_1+P_2$, where $P_1,P_2\in \mathbb C\langle X_1,\dots,X_k\rangle$, 
and 
{that}
linear polynomials
$$
L_{j}=\begin{bmatrix}
0 & u_j\\
v_j & Q_j
\end{bmatrix}\in M_{n_j}(\mathbb C\langle X_1,\dots,X_k\rangle),\quad j=1,2,
$$ 
with the desired properties have been found for $P_1$ and $P_2$.  Then we set $n=n_1 +n_2 -1$ and observe that the matrix 
$$L=
\begin{bmatrix}
0 & u_1 & u_2\\
v_1 & Q_1 & 0\\
v_2 & 0 &Q_2
\end{bmatrix}=
\begin{bmatrix}
0 & u\\
v & Q
\end{bmatrix}
\in M_{n_1+n_2-1}(\mathbb C\langle X_1,\dots X_k\rangle).
$$
is a linearization of $P_1+P_2$ with the desired properties. 
The construction of a linearization is now easily 
 completed for an arbitrary polynomial. 
 Suppose now that $P$ is a selfadjoint polynomial, so  $P=P_0+P_0^*$ for some 
other polynomial $P_0$. Suppose that the matrix 
$$
\begin{bmatrix}
0 & u_0\\
v_0 & Q_0
\end{bmatrix}.
$$  
of size $n_0$ is a linearization of $P_0$. Then we set $n=2n_0 -1$ and observe that the selfadjoint linear polynomial
$$
\begin{bmatrix}
0 & u_0 & v_0^*\\
u_0^* & 0 & Q_0^* \\
v_0 & Q_0 & 0
\end{bmatrix}=\begin{bmatrix}
0 & u\\
u^* & Q
\end{bmatrix}
$$
linearizes $P$. It is easy to verify inductively that this 
construction produces a matrix $Q$ such that the constant term of $Q^{-1}$ has 
spectrum contained in $\{1,-1\}$. 
These properties of $Q$  \cite{Mai}, and particularly the following observation,
facilitate our analysis.

\begin{lemma}\label{Q invers} Let $P\in \mathbb C\langle X_1,\dots,X_k\rangle$, 
and let 
$$
L=\begin{bmatrix}
0 & u\\
v & Q
\end{bmatrix}\in M_n(\mathbb C\langle X_1,\dots,X_k\rangle)
$$
be a linearization of $P$ as constructed above. There exist a permutation matrix
$T\in M_{n-1}$ and a strictly lower triangular matrix $N\in M_{n-1}(\mathbb C\langle X_1,\dots,X_k\rangle)$ such that
$Q^{-1}=T(1_{n-1}+N)$.
\end{lemma}

\begin{proof} We show that there exist
	a permutation matrix
	$T_0\in M_{n-1}$ and a strictly lower triangular matrix
	permutation matrices $N_0\in M_{n-1}(\mathbb C\langle X_1,\dots,X_k\rangle)$ such that
	$Q=(1_{n-1}-N_0)T_0$. Then we can define $T=T_0^{-1} $ and $N=\sum_{j=1} ^{n-2}N_0^j$. 
The existence of $T_0$ and $N_0$ is proved by following inductively the construction of $L$. If 
$P=X_{i_1}\cdots X_{i_\ell}$, $\ell\ge2$, we define
$$T_0=\begin{bmatrix}
0 & \cdots  & 1\\
\vdots& \reflectbox{$\ddots$}&\vdots\\
1 & \cdots&0
\end{bmatrix}, $$ 
and let the only nonzero entries of $N_0$ be $X_{i_2},\dots,X_{i_\ell}$ just below the main diagonal. If 
$P=P_1+P_2$, and linearizations for $P_1$ and $P_2$ have been found, then the desired matrices are 
obtained simply by taking direct sums of the matrices corresponding to $P_1$ and $P_2$.  The case in 
which $P=P_0 +P_0^*$ is treated similarly (different factorizations must be used for $Q_0$ and $Q_0^*$).
\end{proof}

\begin{lemma}\label{o lema}
Suppose that $P\in \mathbb C\langle X_1,\dots,X_k\rangle$, and let 
$$
L=\begin{bmatrix}
0 & u\\
v & Q
\end{bmatrix}\in M_n(\mathbb C\langle X_1,\dots,X_k\rangle)
$$
be a linearization of $P$ with the properties outlined above. 
Then for every $N\in\mathbb N$, and for every $S_1,\dots,S_k\in M_N(\mathbb C)$, 
we have
$$
\det (ze_{1,1}\otimes I_N-L(S_1,\dots,S_k))= \pm \det (zI_n-P(S_1,\dots,S_k)),
$$
where the sign is $\det(Q(S_1,\dots,S_k))$. Moreover,
$$\dim\ker(zI_n-P(S_1,\dots,S_k))=\dim\ker(ze_{1,1}\otimes I_N-L(S_1,\dots,S_k))
\quad z\in\mathbb C.$$
\end{lemma}

\begin{proof}
Suppressing the variables $S_1,\dots,S_k$, we have
$$
\begin{bmatrix}
1  & -uQ^{-1}\\
0 & 1_{n-1}
\end{bmatrix}
\begin{bmatrix}
z  & -u\\
-v & -Q
\end{bmatrix}
\begin{bmatrix}
1  & 0\\
-Q^{-1}v & 1_{n-1}
\end{bmatrix}=
\begin{bmatrix}
z-P  & 0\\
0 & -Q
\end{bmatrix},\quad z\in\mathbb C.
$$
Lemma \ref{Q invers} implies that $\det Q(S_1,\dots,S_k)$ is $\pm1$ and the determinant identity follows 
immediately. The dimension of the kernel of a square matrix does not change if the matrix is 
multiplied by some other invertible matrices.  Also, since $Q$ is invertible, the 
kernel of the matrix on the right hand side of the last equality is easily 
identified with $\ker(z-P)$.  The last assertion follows from these observations.
\end{proof}

In the case of selfadjoint polynomials, applied to selfadjoint matrices, we 
can estimate how far $ze_{1,1}-L$ is from 
not being invertible. 

\begin{lemma}\label{alta lema}
Suppose that $P=P^*\in \mathbb C\langle X_1,\dots,X_k\rangle$, and let 
$$
L=\begin{bmatrix}
0 & u^*\\
u & Q
\end{bmatrix}\in M_n(\mathbb C\langle X_1,\dots,X_k\rangle)
$$
be a linearization of $P$ with the properties outlined above. There exist
polynomials $T_1,T_2\in\mathbb C[ X_1,\dots,X_k]$ with nonnegative coefficients 
with the following property: given arbitrary selfadjoint elements $S_1,\dots,S_k$ in a 
unital $C^*$-algebra $\mathcal A$, and given  $z_0\in\mathbb C$ such that 
$z_0-P(S)$ is invertible, we have
$$
\left\|(z_0e_{1,1}-L(S))^{-1}\right\|\leq 
T_1\left(\|S_1\|,\dots,\|S_k\|\right) \left\|(z_0-P(S))^{-1}\right\|+ 
T_2\left(\|S_1\|,\dots,\|S_k\|\right).
$$
In particular, given two real constants $C,\delta>0$, there exists $\varepsilon>0$ such that  
${\mathrm{dist}}(z_0,\sigma(P(S)))\geq \delta$ and $\|S_1\|+\cdots+\|S_k\|\leq C$ imply   
$\mathrm{dist}(0,\sigma(z_0e_{1,1}-L(S)))\ge\varepsilon$.
\end{lemma}

\begin{proof}
For every element $a$ of a $C^*$-algebra, we have 
${\rm{dist}}(0,\sigma(a))\ge1/\|a^{-1}\|$. 
Equality is achieved, for instance, if $a=a^*$. A matrix calculation
(in which we suppress the variables $S$) shows that
$$
(z_0e_{1,1}-L)^{-1}=
\begin{bmatrix}
1  & 0\\
-Q^{-1}u & 1_{n-1}
\end{bmatrix}
\begin{bmatrix}
(z_0-P)^{-1}  & 0\\
0&-Q^{-1}
\end{bmatrix}
\begin{bmatrix}
1  & -u^*Q^{-1}\\
0 & 1_{n-1}
\end{bmatrix}.
$$
The lemma follows now because the entries of $u(S)$, $u^*(S)$, and $Q(S)^{-1}$ 
are polynomials in $S$, and
 $$\|(z_0-P(S))^{-1}\|=1/{\rm{dist}}(z_0,\sigma(P(S)))$$ because $P(S)$ 
 is selfadjoint.
\end{proof}
The dependence on $L$ in the above lemma is given via the norms of $Q^{-1}$ and of $u$. 
Since $\lim_{z\to\infty}\|(ze_{1,1}-L(S))^{-1}\|\ne0$, we see that $T_2 \ne0$.

\section{Subordination}\label{sec:subordination}

Consider a von Neumann algebra $\mathcal M$ endowed with a normal faithful tracial state $\tau$, let 
$\mathcal B\subset\mathcal N\subset\mathcal M$ be unital von Neumann subalgebras, and denote by 
$E_\mathcal N:\mathcal M\to\mathcal N$ the unique trace-preserving conditional expectation of  
$\mathcal M$ onto $\mathcal N$ (see \cite[Proposition V.2.36]{Takesaki1}). Denote by $\mathbb H^+
(\mathcal M)$ the operator upper-half plane of $\mathcal M$: $\mathbb H^+
(\mathcal M)=\{x\in\mathcal M\colon\Im x:=(x-x^*)/2i>0\}$. Given two arbitrary selfadjoint 
elements $c,d\in\mathcal M$, we define the open set
$\mathcal G_{c,d,\mathcal B,\mathcal N}$
to consist of those elements $\beta\in\mathcal  B$ such that $\beta-(c+d)$ is invertible and 
$E_\mathcal N((\beta-(c+d))^{-1})$ is invertible as well. Then the function
\begin{equation*}
\omega_{c,d,\mathcal B,\mathcal N}:\mathcal G_{c,d,\mathcal B,\mathcal N}\to\mathcal M
\end{equation*}
defined by
\begin{equation}
\omega_{c,d,\mathcal B,\mathcal N}(\beta)=c+[E_\mathcal N((\beta-(c+d))^{-1})]^{-1},\quad \beta\in\mathcal G_{c,d,\mathcal B,\mathcal N},
\label{definition of omega cdBN}
\end{equation}
is analytic. This equation can also be written as
\begin{equation}
E_\mathcal N((\beta-(c+d))^{-1})=
(\omega_{c,d,\mathcal B,\mathcal N}(\beta)-c)^{-1}\quad \beta\in\mathcal G_{c,d,\mathcal B,\mathcal N}.\label{subordination equation}
\end{equation}
Properties (1), (2), and (3) in the following lemma are easy observations, while (4) follows as in 
\cite[Remark 2.5]{BPV}.

\begin{lemma}\label{preliminary}
	Fix $\mathcal B\subset\mathcal N\subset\mathcal M$ and $c,d\in\mathcal M$  as above. Then:
	\begin{enumerate}
		\item The set $\mathcal G_{c,d,\mathcal B,\mathcal N}$ is selfadjoint.
		\item 
		$
		\omega_{c,d,\mathcal B,\mathcal N}(\beta^*)=\omega_{c,d,\mathcal B,\mathcal N}(\beta)^*$, 
$\beta\in\mathcal G_{c,d,\mathcal B,\mathcal N}.
		$
		\item $\mathbb H^+(\mathcal B)\subset\mathcal  G_{c,d,\mathcal B,\mathcal N}$ and 
		$\omega_{c,d,\mathcal B,\mathcal N}
		(\mathbb H^+(\mathcal B))\subset \mathbb H^+(\mathcal N)$.
		\item $\Im(\omega_{c,d,\mathcal B,\mathcal N}(\beta))\ge\Im(\beta)$, $\beta\in\mathbb H^+(\mathcal B)$.
	\end{enumerate}
\end{lemma}
There is one important case in which $\omega_{c,d,\mathcal B,\mathcal N}$ takes values in $\mathcal B$, 
and thus (\ref{subordination equation}) allows us to view $\omega_{c,d,\mathcal B,\mathcal N}|\mathbb 
H^+(\mathcal B)$ as a subordination function in the sense of Littlewood. Denote by 
$\mathcal G_{c,d,\mathcal B,\mathcal N}^0$ the connected component of $\mathcal G_{c,d,\mathcal B,
\mathcal N}$ that contains $\mathbb H^+(B)$. The following basic result is from \cite{V2}.

\begin{thm}\label{omega matrix valued}
	With the above notation, suppose that $c$ and $d$ are free over $\mathcal B$ and $\mathcal 
N=\mathcal B\langle c\rangle$ is the unital von Neumann generated by $\mathcal B$ and $c$. Then 
	$$
\omega_{c,d,\mathcal B,\mathcal N}(\mathcal G^0_{c,d,\mathcal B,\mathcal N})\subset\mathcal  B.
$$
\end{thm}

In our applications, the algebra $\mathcal B$ is (isomorphic to) $M_n(\mathbb C)$ 
for some $n\in\mathbb N$. More precisely, let  $\mathcal M$ be a von Neumann algebra endowed with 
a normal faithful tracial state $\tau$, and let $n\in\mathbb N$. Then $M_n (\mathbb C)$ can be identified 
with the subalgebra $M_n(\mathbb C)\otimes1$ of $M_n ({\mathcal M})=M_n(\mathbb C)\otimes 
\mathcal M$. Moreover, $M_n({\mathcal M})$ is endowed with the faithful 
normal tracial state $\mathrm{tr}_n\otimes\tau=(1/n)\mathrm{Tr}_n\otimes\tau$,
and $\mathrm{Id}_{M_n(\mathbb C)}\otimes\tau$ is the trace-preserving conditional expectation from 
$M_n ({\mathcal M})$ to $M_n(\mathbb C)$. The following result is from \cite{NSS}.

\begin{prop}\label{fr}
	Let $\mathcal{M}$ be a von Neumann algebra endowed 
	with a normal
	faithful tracial state $\tau$, 
	let $c,d\in\mathcal M$ be freely independent, let $n$ be a 
	positive integer, 
	and let $\gamma_1,\gamma_2 \in M_n(\mathbb{C})$. Then 
	$\gamma_1\otimes c$ and $\gamma_2\otimes d$  are free over $M_n (\mathbb C)$.
\end{prop}

We show next how the spectrum of  $P(c,d)$ relates with the functions $\omega$ 
defined above. Thus, we fix $P=P^*\in\mathbb C\langle X_1,X_2\rangle$
and a linearization $L=\gamma_0\otimes1+\gamma_1\otimes X_1
+\gamma_2\otimes X_2$ of $P$ as constructed in Section \ref{sec:linearization}. Thus,  $\gamma_0,
\gamma_1,\gamma_2\in M_n(\mathbb{C})$ are selfadjoint matrices for some $n\in\mathbb N$. (Clearly
 $\gamma_1 \ne0$ unless $P\in\mathbb C\langle X_2\rangle$.)
 Then we consider the random variables $\gamma_1\otimes c$ and $\gamma_2\otimes d$ in 
$M_n(\mathcal M)$, the algebra $\mathcal B=M_n(\mathbb C)\subset M_n (\mathcal M)$, 
and $\mathcal N=M_n(\mathbb C\langle c\rangle)$; clearly $\mathcal B\subset\mathcal N\subset 
M_n(\mathcal M)$. We set 
$$
\mathcal G=\mathcal G^0_{\gamma_1\otimes c,\gamma_2\otimes d,\mathcal B,\mathcal N}
$$ 
and
$$
\omega=\omega_{\gamma_1\otimes c,\gamma_2\otimes d,\mathcal B,\mathcal N}
\colon\mathcal G\to \mathcal N.
$$
Thus,
$$E_{\mathcal N}\left[(\beta\otimes1-\gamma_1\otimes c-\gamma_2\otimes d)^{-1}\right]
=(\omega(\beta)\otimes1-\gamma_1\otimes c)^{-1},\quad \beta\in \mathcal G.$$ 
The left hand side of this equation is 
defined if $\beta=ze_{11}-\gamma_0$ for 
some $z\in\mathbb{C}\setminus \sigma(P(c,d))$, 
so it would be desirable that 
$ze_{11}-\gamma_0\in \mathcal G$ for such values 
of $z$. 
This is not true except for special cases. (One such case applies to $P=X_1+X_2$ if $d$ is a semicircular 
variable free from $c$ \cite{biane,CAOT}.) The following lemma offers a partial result.

\begin{lemma}\label{large |z|} 
	With the notation above, there exists $k>0$ depending only on $L$, $\|c\|$, and $\|d\|$ such that 
	$ze_{1,1}-\gamma_0\in\mathcal G$ if
	$|z|>k$. The analytic function $u(z)=\omega(ze_{1,1}-\gamma_0)$ satisfies the equation $u(\overline{z})=u(z)^*$ for $|z|>k$.
\end{lemma}

\begin{proof}
Define an analytic function $F:\mathbb{C}\setminus\sigma(P(c,d))\to \mathcal N$ by 
\begin{equation}
F(z)=E_{\mathcal N}\left[((ze_{11}-\gamma_0)\otimes1-\gamma_1\otimes c-\gamma_2\otimes d)^{-1}\right],\quad
z\in\mathbb{C}\setminus\sigma(P(c,d)).
\label{functia F}
\end{equation}
We show that $F(z)$ is invertible if $|z|$ is sufficiently large. Suppressing the variables $c$ and $d$ 
from the notation, it follows from the factorization used in the proof of Lemma \ref{o lema} that
$$
F(z)=E_{\mathcal N}\begin{bmatrix}
(z-P)^{-1} & -(z-P)^{-1}u^*Q^{-1} \\
-Q^{-1}u(z-P)^{-1} & Q^{-1}u(z-P)^{-1}u^*Q^{-1}-Q^{-1}
\end{bmatrix}.
$$
Moreover, because of the matrix structure of $\mathcal N$, this matrix can be obtained by applying 
$E_{\mathbb C\langle c\rangle}$ entrywise. According to the Schur complement formula, a matrix
$\begin{bmatrix}
A&B\\
C&D
\end{bmatrix}$ is invertible if both $A$ and $D-BA^{-1}C$ are invertible. For our matrix, we have
$A=A(z)=
E_{\mathbb C\langle c\rangle}((z-P)^{-1})$. The fact that
$\|zA(z)-1\|<1$ for $|z|>2\|P\|$ implies that $A(z)$ is invertible. Next, we see that $\|(z-P)^{-1}\|$ and $\|A(z)^{-1}\|$ are comparable to $1/|z|$ and $|z|$, respectively. Using these estimates, one sees also that
$\lim_{|z|\to\infty }\|B(z)A(z)^{-1}C(z)\|=0$, so the invertibility of $F(z)$ would follow from the invertibility of 
$D(z)$ for large $|z|$. Since $$\lim_{|z|\to\infty}\|D(z)+E_{M_{n-1}
	(
{\mathbb C\langle c\rangle})}(Q^{-1})\|=0,$$
we only need to verify that $E_{M_{n-1}
	(
{\mathbb C\langle c\rangle})}(Q^{-1})$
is invertible.  Write $Q^{-1}=T(1_{n-1}+N)$ as in Lemma \ref{Q invers}. We have 
$$E_{M_{n-1}
	(
{\mathbb C\langle c\rangle})}(Q^{-1})=T(1_{n-1}+E_{M_{n-1}
	(
{\mathbb C\langle c\rangle})}(N)),
$$
and $  E_{M_{n\!-\!1}
	(
{\mathbb C\langle c\rangle})}(N)$ is strictly lower triangular. The invertibility of $E_{M_{n\!-\!1}
	(
{\mathbb C\langle c\rangle})}(Q^{-1})$ follows. The quantities $\|D(z)+E_{M_{n-1}
	(
{\mathbb C\langle c\rangle})}(Q^{-1})\|$ and $\|Q^{-1}\|$ can be estimated using only $L$, $\|c\|$, and 
$\|d\|$, and this shows that $k$ can be chosen as a function of these objects. The last assertion of the 
lemma is immediate.
\end{proof}

The estimates in the preceding proof apply, by virtue of continuity, to nearby points in
$M_n(\mathbb C)$.  We record the result for later use.

\begin{cor}\label{cor:local bounds for omega}
	Let $c,d$ and $k$ be as in Lemma {\em\ref{large |z|}} and let $z\in\mathbb C\setminus[-k,k]$. 
Then there exist a constant $k'>0$ and a neighborhood $W$ of $ze_{1,1}-\gamma_0$, depending only on 
$\|c\|,\|d\|$, and $L$, such that $V\subset\mathcal G$ and $\|\omega(\beta)\|\le k'$ for $\beta\in W$.
\end{cor}

In some cases of interest, the analytic function $u$ extends to the entire upper and lower half-planes. We 
recall that a function $v$ defined in a domain $G\subset\mathbb C$ with values in a Banach space 
$\mathcal X$ is said to be meromorphic if, for every $z_0\in G$, the function $(z-z_0)^nv(z)$ is analytic 
in a neighborhood of $z_0$ for sufficiently large $n$. For instance, if $\mathcal X$ is a finite dimensional 
Banach algebra and $h:G\to\mathcal X$ is an analytic function such that $h(z)$ is invertible for some 
$z\in G$, then the function $v(z)=h(z)^{-1}$ is meromorphic in $G$. This fact follows easily once we 
identify $\mathcal X$ with an algebra of matrices, so the inverse can be calculated using determinants.

\begin{lemma}\label{meromorphic finite dim}
The function $u$ defined in Lemma 
{\em\ref{large |z|}} is meromorphic in $\mathbb C\setminus\sigma(P(c,d))$.
\end{lemma}

\begin{proof}
	The lemma follow immediately from the observation preceding the statement applied to the function 
$F$ defined in (\ref{functia F}) which is analytic in $\mathbb C\setminus\sigma(P(c,d))$
	since, by hypothesis and by Theorem \ref{omega matrix valued}, $u$ takes values in a finite 
dimensional algebra. 
\end{proof}

The 
conclusion of the 
preceding lemma applies, for instance, when  $\mathcal M=M_N\otimes L^\infty(\Omega)$ with 
the usual trace $\mathrm{tr}_N\otimes\mathbb E$. This 
situation arises in the study of random matrices.  The function $u$ is also meromorphic provided that $c$ 
and $d$ are free random variables and $\mathcal N=M_n(\mathbb C)\langle\gamma_1
\otimes c\rangle$.  Since $\gamma_1\ne0$, we have $\mathcal N=M_n(\mathbb C\langle c\rangle)$.

\begin{lemma}\label{meromorphic free case}
	If $c$ and $d$ are free and $\mathcal N=M_n(\mathbb C)\langle\gamma_1
	\otimes c\rangle $, then the function $u$ defined in Lemma {\em\ref{large |z|}} is meromorphic in 
$\mathbb C\setminus\sigma(P(c,d))$ with values in $M_n(\mathbb C)$. Moreover, given an arbitrary 
$\lambda\in\sigma(c)$, the function 
	$(u-\lambda\gamma_1)^{-1}$ extends analytically to
	$\mathbb C\setminus\sigma(P(c,d))$.
\end{lemma}

\begin{proof}	
	Theorem \ref{omega matrix valued} shows that $\omega$ takes values in $M_n(\mathbb C)$. We 
have  established that the domain $\mathcal G$ of $\omega$ contains $z\otimes e_{11}-\gamma_0$ for 
sufficiently large $|z|$. Fix a character $\chi$ of the commutative $C^*$-algebra $\mathbb{C}\langle 
c\rangle$ and denote by
	$\chi_n:M_n(\mathbb{C})\langle \gamma_1\otimes c\rangle\to M_n(\mathbb{C})$ the algebra 
homomorphism obtained by applying $\chi$ to each entry. Using the notation (\ref{functia F}), we have
	$$(\omega(z\otimes e_{11}-\gamma_0)-\gamma_1\otimes \chi(c))\chi_n(F(z))=I_n$$
	for sufficiently large $|z|$. It follows immediately that the function
	$$
	u_1 (z)=\gamma_1 \otimes\chi(c)+[\chi_n(F(z))]^{-1}
	$$
	is a meromorphic continuation of $u$  to 
	$\mathbb{C}\setminus\sigma(P(c,d))$. Moreover, the equation
	\begin{equation}\label{u1}
	(u_1(z)-\gamma_1\otimes c)F(z)=I_n
	\end{equation}
	holds for large values $|z|$ and hence it holds on the entire domain of analyticity of $u_1$.  It 
follows that $\mathcal G$ contains $z\otimes e_{11}-\gamma_0$ whenever $z\in\mathbb{C}\setminus 
\sigma(P(c,d))$ is not a pole of $u_1$, and $\omega(z\otimes e_{11}-\gamma_0)=u_1(z)$ for such values 
of $z$. To verify the last assertion, choose $\chi$
	such that $\chi(c)=\lambda$ and apply $\chi_n$ to (\ref{u1}) to obtain
	$$
	(u_1(z)-\lambda\gamma_1)\chi_n(F(z))=I_n.
	$$
	Thus $\chi_n\circ F$ is an analytic extension of $(u-\lambda\gamma_1)^{-1}$ to
	$\mathbb C\setminus\sigma(P(c,d))$.
\end{proof}

\section{Main results and example}\label{re}

Fix a polynomial $P=P^*\in\mathbb C\langle X_1,X_2\rangle$ and choose, as in 
Section \ref{sec:linearization}, a linearization of $P$ of the form $ze_{1,1}-L$, where 
$L=\gamma_0\otimes 1+\gamma_1\otimes X_1+\gamma_2\otimes X_2\in 
M_n(\mathbb C\langle X_1,X_2\rangle)$. In particular, 
$\gamma_0,\gamma_1,\gamma_2\in M_n(\mathbb C)$ are selfadjoint matrices.

Suppose that $\{A_N\}_{N\in\mathbb N}$ and $\{B_N\}_{N\in\mathbb N}$
are two  sequences of selfadjoint random matrices satisfying the 
hypotheses (A1)--(A3) and (B0)--(B2) of Section \ref{model}.
As noted earlier,  the pairs $(A_N,B_N)$ 
in $M_N (\mathbb{C})$ converge almost surely in 
distribution to a pair $(a,b)$ of  freely independent selfadjoint random 
variables in a $C^*$-probability space $(\mathcal A,\tau)$ such that 
$\mu_a=\mu$ and $\mu_b=\nu$. 
By Theorem \ref{omega matrix valued}, there exists a selfadjoint open set $\mathcal G\subset M_n(\mathbb C)$, and an analytic function
$\omega\colon \mathcal G\to
 M_n(\mathbb C)$ such that 
$$
(\omega(\beta)\otimes 1-\gamma_1\otimes a)^{-1}
 = E_{M_n(\mathbb C\langle a\rangle)}\left[(\beta\otimes 1-(\gamma_1 \otimes a+
\gamma_2 \otimes  b))^{-1}\right],\quad\beta\in \mathcal G.$$
As shown in Lemma \ref{meromorphic free case},
the map  
$$
u(z)=\omega(z\otimes e_{1,1}-\gamma_0)
$$ 
is meromorphic on $\mathbb C\setminus\sigma(P(a,b))$. Define a new function 
$$u_0(z)=(u(ze_{1,1}-\gamma_0)+iI_n)^{-1}.$$ It follows from Lemma \ref{preliminary}
 that $u_0$ continues analytically to a neighbourhood of $\mathbb R\setminus\sigma(P(a,b))$. (Indeed, $u_0$ is bounded near every pole of $u$.) Define 
$$H_j(z)=\det[(\theta_j\gamma_1+i)u_0(z)-I_n],\quad j=1,\dots,p$$ and denote by $m_{j}(t)$ the order
of $t$ as a zero of $H_j(z)$ at $z=t$. Also set
$m(t)=m_1(t)+\cdots+m_p(t)$ for $t\in\mathbb R\setminus\sigma(P(a,b))$, and note that $\{t:m(t)\ne0\}$ is an isolated set in $\mathbb R\setminus\sigma(P(a,b))$.
With this notation, we are ready to state our first main result. The notation $E_{A_N}$ indicates the spectral measure of the matrix $A_N$, that is, $E_{A_N}(S)$ is the orthogonal projection onto the linear 
span of the eigenvectors of $A$ corresponding to eigenvalues in the Borel set $S$.

\begin{thm}\label{Main}  
{\rm(1)} Suppose that $t\in\mathbb R\setminus\sigma(P(a,b))$. Then there exists $\delta_0>0$ such that for every $\delta
\in(0,\delta_0)$, almost surely for large $N$, the random matrix $P(A_N,B_N)$ has exactly $m(t)$ 
eigenvalues in the interval $(t-\delta,t+\delta)$, counting multiplicity.

{\rm(2)} Suppose in addition that the spikes of $A_N$ 
{are distinct}
and $\det H_{i_0}(t)=0$. Then, for $\varepsilon$ 
small enough, almost surely 
\begin{equation}\label{vecteur}
\lim_{N\to\infty}\left\|E_{A_N}(\{\theta_i\})\left[E_{P(A_N,B_N)}((t-\varepsilon,t+\varepsilon))
-\delta_{i,i_0}\mathcal C_i(t)I_N\right]E_{A_N}(\{\theta_i\})\right\|=0,
\end{equation}
where $\mathcal C_i(t)=
\lim_{z\to t}(z-t)\left[(u(z)-\theta_i\gamma_1)^{-1}\right]_{1,1}$ is the 
residue of the meromorphic function 
$\left[(u(z)-\theta_i\gamma_1)^{-1}\right]_{1,1}$ at $z=t$.
\end{thm}

\begin{rem}\label{rmk5.2}
If we know in addition that $\omega$ is analytic at the point $\beta=te_{1,1}-\gamma_0$, then 
the function $H_j(z)$ can be replaced by $z\mapsto\det[\theta_j\gamma_1-u(z)].$
In that case, $m(t)$ is equal to the multiplicity of $t$ as a zero of 
$$
z\mapsto\prod_{j=1}^p \det[\theta_j\gamma_1-u(z)].
$$
This situation arises, for instance, when $b$ is a semicircular variable  and it is relevant when $B_N$ is 
replaced by a Wigner matrix $X_N/\sqrt{N}$. Under the hypotheses $(X0)-(X3)$ of Section \ref{model}, 
we obtain the following result. Note that the  subordination function $\omega$ has the more explicit form
$$
\omega(\beta)=\beta-\gamma_2(\mathrm{Id}_{M_n(\mathbb C)}\otimes\tau)
\left[(\beta\otimes 1_{\mathcal A}-\gamma_1\otimes a -\gamma_2\otimes b)^{-1}\right]\gamma_2,
\quad\beta\in \mathcal G.
$$
\end{rem}
\begin{thm}\label{thm8.2}
Let $a$ and 
$b$ be free selfadjoint elements  in a ${C}^*$-probability space 
$({\mathcal A},\tau)$  with distribution $\mu$ and $\nu_{0,1}$ respectively $($see \eqref{demicercle}$)$, 
$t\in\mathbb R\setminus\sigma({P(a,b)})$, and let $m(t)$ be defined as in Remark {\em \ref{rmk5.2}}. 
Then, for sufficiently small $\varepsilon$, almost surely for large $N$, there are exactly $m(t)$ 
eigenvalues of $P(A_N,{X_N}/{\sqrt{N}})$  in an $\varepsilon$-neighborhood of $t$.
\end{thm}

\begin{rem}\label{calcul}
The subordination function can be calculated more explicitly if $\mu=\delta_0$ (and hence $a=0$). 
In this case, 
$$ 
\omega(\beta)=
\left\{({\rm Id}_{M_n(\mathbb C)}\otimes\tau)
\left[(w\otimes1-\gamma_2\otimes b)^{-1}\right]\right\}^{-1}.
$$
\end{rem}

As an illustration, consider the random matrix 
$$
M=A_N \frac{X_N}{\sqrt{N}} +\frac{X_N}{\sqrt{N}}A_N +\frac{X_N^2}{N},
$$
where $X_N$ is a standard G.U.E. matrix of size $N$ 
(thus, each entry of $X_N$ has unit norm in $L^2(\Omega)$) and 
$$
A_N=\text{Diag}(\theta, 0,\ldots,0),\quad \theta\in\mathbb{R}\setminus\{0\}. 
$$
In this case, $A_N$ has rank one, and thus $\mu=\delta_0$. It follows that the limit spectral 
measure $\rho$ of $M$ is the same as the limit spectral measure of
$X_N^2/{N}$. Thus, $\eta$ is the Marchenko-Pastur distribution $\rho$ with parameter 1:
$$ 
{\rm d}\eta(x) =\frac{\sqrt{\left(4-x\right)x}}{2\pi  x}1_{(0,4)}(x)\,{\rm d}x.
$$
The polynomial $P$ is $P(X_1,X_2) =X_1X_2 +X_2X_1 + X_2^2$, $\mu=\delta_0$, and $\nu$ is 
the standard semi-circular distribution. An economical linearization of $P$ is provided by  
$L=\gamma_0 \otimes 1 + \gamma_1 \otimes X_1+\gamma_2\otimes X_2$,  where
$$
\gamma_0= 
\begin{bmatrix} 0 & 0&0\\
0& 0 & -1 \\
0&-1&0
\end{bmatrix},\quad \gamma_1=
\begin{bmatrix} 
0 & 0&1\\
0& 0 & 0\\
1&0&0
\end{bmatrix},\quad\gamma_2=
 \begin{bmatrix} 
0 & 1&\frac{1}{2}\\
1& 0 & 0\\ 
\frac{1}{2}&0&0
\end{bmatrix}.
$$
Denote by 
$$
G_\eta(z)=\int_0^4\frac1{z-t}\,{\rm d}\eta(t)=\frac{z-\sqrt{z^2-4z}}{2z},\quad z\in\mathbb C
\setminus[0,4]
$$
the Cauchy transform of the measure $\eta$. 
{(The branch of the square root is chosen so $\sqrt{z^2-4z}>0$ for $z>4$.)}
This function satisfies the 
quadratic equation $zG_\eta(z)^2-zG_\eta(z)+1=0$. Suppose now that $x\notin [0,4]$. Denoting by 
$E={\rm Id}_{M_3(\mathbb C)}\otimes\tau\colon M_3(\mathcal A)\to M_3(\mathbb C)$ 
the usual expectation and using Remark \ref{calcul}, we have
$$
\omega(xe_{1,1}-\gamma_0)=
E\left[(xe_{1,1}-\gamma_0-\gamma_2\otimes b)^{-1}\right]^{-1},
\quad x\in\mathbb R\setminus[0,4].
$$
The inverse of $(xe_{1,1}-\gamma_0)\otimes1-\gamma_2\otimes b$ is then calculated 
explicitly and application of the expected value to its entries yields eventually
$$
\omega(xe_{11}-\gamma_0)= 
\begin{bmatrix} 
\frac{1}{G_\eta (x)} &0&0 \\
0 &\frac{1}{xG_\eta (x)}-1 & \frac{1}{2x G_\eta (x)} +\frac{1}{2}\\
0&\frac{1}{2xG_\eta (x)}+\frac{1}{2}&\frac{1}{4xG_\eta (x)} -\frac{1}{4} 
\end{bmatrix}.
$$ 
After calculation, the equation $\det[\gamma_1\theta-\omega(xe_{11}-\gamma_0)]=0$ 
reduces to 
\begin{equation} 
\theta^2 G_\eta (x)^2 -(1-G_\eta (x))=0.\label{equation}
\end{equation}
This equation has  two solutions, namely 
$$
\frac{2\theta^4}{-(3\theta^2+1)\pm\sqrt{4\theta^2+1} (\theta^2 +1)},
$$ 
one of which is negative. The positive solution belongs to $[4,+\infty )$ precisely when 
$|\theta|>\sqrt 2$. Thus, the matrix $M_N$ exhibits one (negative) outlier when 
$0<|\theta|\le\sqrt 2$ and two outliers (one negative and one $>4$)  when 
$|\theta|>\sqrt 2$.  The second situation  is illustrated by the simulation presented in Figure \ref{figgy}.
\begin{figure}
\centering
\includegraphics[width=0.99\textwidth]{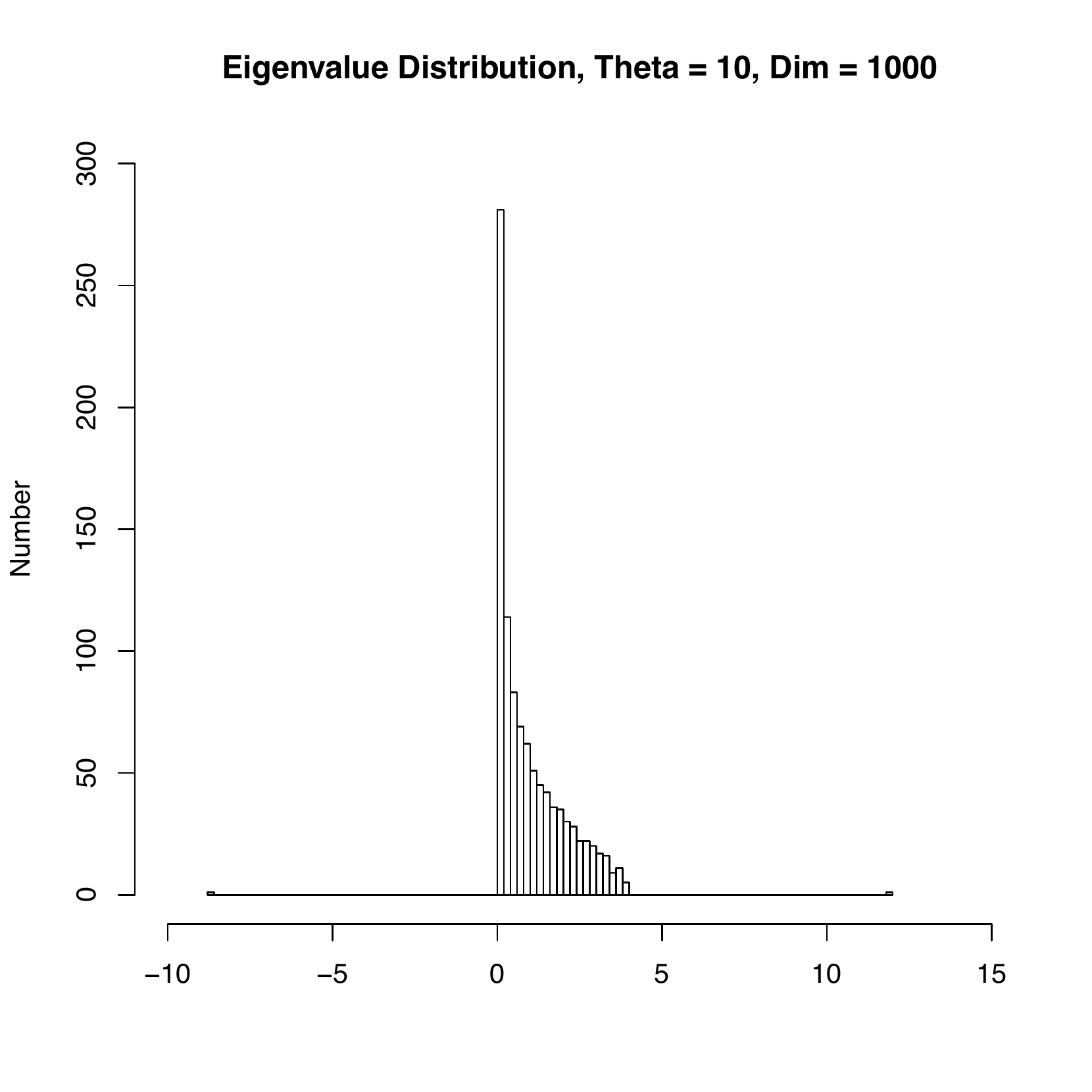} 
\caption{One sample from the model described in remark 
\ref{calcul} corresponding to $\theta=10$, with matrix size 
$N=1000$.}\label{figgy}
\end{figure}

\section{{Outline} of the proofs}\label{sp}
We consider first the matricial model (\ref{uni}), that is, $Z_N=P(A_N,B_N)$, where $A_N$ and $B_N$ are 
independent and the distribution of $B_N$ is invariant under unitary conjugation. As seen in 
\cite[Proposition 6.1]{CM}, $B_N$ can be written as $B_N = U_ND_N U_N^*$ almost surely, where
$U_N$ is distributed according to the Haar measure on the unitary group ${\rm U}(N)$,
$D_N$ is a diagonal random matrix, and
$U_N$ is independent from $D_N$. As pointed out in \cite[Section 3, Assertion 2]{M}, it suffices to prove 
Theorem \ref{Main} under the assumption that $A_N$ and $D_N$ are constant selfadjoint matrices that 
can be taken to be diagonal in the standard basis. Thus, we work with
$$
A_N=\text{Diag}(\lambda_1(A_N),\dots,\lambda_N(A_N))
$$
and
$$
B_N=U_ND_NU_N^*,\quad D_N=\text{Diag}
(\lambda_1(D_N),\dots,\lambda_N(D_N))
$$
where $\lambda_j(A_N)=\theta_j,1\le j\le p$, 
and $U_N$ is uniformly 
distributed in ${\rm U}(N)$.

Similarly, the proof for the second model $Z_N=P(A_N,X_N/\sqrt{N})$ reduces to the special case in which 
$A_N$ is a constant matrix.
 
Choose a linearization $L$ of $P$ as in Section \ref{sec:linearization}.  In the spirit of \cite{BGR}, the first 
step in the proofs of Theorems \ref{Main} and \ref{thm8.2} consists of reducing the problem 
to the convergence of random matrix function $F_N$ of fixed size $np$, 
involving the generalized resolvent of the linearization applied to $Z_N$.
For the first model, this convergence is established in Section \ref{cvF} by extending 
the arguments of \cite{BBCF} and making use of the properties of the
operator-valued subordination function described in Section \ref{sec:subordination}. 
For the second model, the convergence 
of $F_N$ is obtained in Section \ref{cvFw} via a comparison with the G.U.E. case.
The case in which $X_N$ is a G.U.E. is, of course,
a particular case of the unitarily invariant model.

\section{Expectations of matrix-valued random analytic maps}\label{cvF}
As seen earlier in this paper, it suffices to prove Theorem \ref{Main} in the special case
in which the matrix $A_N$ is constant and $B_N$ is a random unitary conjugate of another constant 
matrix. In this section, we establish some useful ingredients specific to this situation. We fix sequences 
$\{C_N\}_{N\in\mathbb N}$ and
$\{D_N\}_{N\in\mathbb N}$, where $C_N,D_N\in M_N(\mathbb C)$, and a sequence 
$\{U_N\}_{N\in\mathbb N}$ of random matrices such that $U_N$ is uniformly distributed in the unitary 
group $\mathrm{U}(N)$. We also fix a selfadjoint polynomial
$P\in\mathbb C\langle X_1,X_2\rangle$ and a selfadjoint linearization
$$L=\gamma_0\otimes 1+\gamma_1\otimes X_1+\gamma_2\otimes X_2\in 
M_n(\mathbb C\langle X_1,X_2\rangle)$$
of $P$ as in Section \ref{sec:linearization}.
The random variables $c_N=C_N\otimes 1_\Omega$ and $d_N=U_ND_NU_N^*$ are viewed
as elements of the noncommutative probability
space $(\mathcal M_N,\tau_N)$, where
$\mathcal M_N=M_N(\mathbb C)\otimes L^\infty(\Omega)$
and $\tau_N=\mathrm{tr}_N\otimes\mathbb E$. 
For every $N\in\mathbb N$ we consider the elements
$\gamma_1\otimes c_N,\gamma_2\otimes d_N\in M_n(\mathbb C)\otimes\mathcal M_N$, and the 
algebras $\mathcal B_N=M_n(\mathbb C)\otimes I_N$ 
and $\mathcal N_N=M_n(\mathbb C)\otimes M_N(\mathbb C)$, both identified with subalgebras of
$M_n(\mathbb C)\otimes\mathcal M_N$. The conditional expectation 
$E_{\mathcal N_N}:M_n(\mathbb C)\otimes\mathcal M_N\to\mathcal N_N$
is simply the expected value and it is accordingly denoted $\mathbb E$.
We use the notation
$$
\mathcal G_N=\mathcal G_{\gamma_1\otimes c_N,\gamma_2 \otimes d_N,\mathcal B_N,\mathcal N_N},
\quad\omega_N=\omega_{\gamma_1\otimes c_N,\gamma_2 \otimes d_N,\mathcal B_N,\mathcal N_N}.
$$
Recall that $\mathcal G_N$ consists of those matrices $\beta\in M_n(\mathbb C)$ with the property that 
$\beta-\gamma_1 \otimes c_N-\gamma_2\otimes d_N$ is invertible in 
$M_n(\mathbb C)\otimes\mathcal M_N$ and
$\mathbb E((\beta-\gamma_1 \otimes c_N-\gamma_2\otimes d_N)^{-1})$ is invertible in 
$\mathcal N_N$. In particular, if $\beta\in\mathcal G_N$ then the matrix $\beta-\gamma_1 \otimes 
C_N-\gamma_2\otimes VD_NV^*$ is invertible for every $V\in{\mathrm U}(N)$. The set $\mathcal G_N$ 
is open and it contains $\mathbb H^+(M_n(\mathbb C))$.  According to Lemma 
\ref{meromorphic finite dim} 
and the remarks following it, the function $u_N(z)=\omega_N(ze_{1,1}-\gamma_0)$ is meromorphic
in $\mathbb C\setminus\sigma(P(c_N,d_N))$.

For simplicity of notation, we write
\begin{equation}
\label{RN}
R_N(\beta)=(\beta-\gamma_1\otimes c_N-\gamma_2 \otimes d_N)^{-1},\quad\beta\in\mathcal G_N,
\end{equation}
and observe that $R_N(\beta)$ is an element of
$M_n(\mathbb C)\otimes M_N(\mathbb C)\otimes L^\infty(\Omega)$, that is, a random matrix of size 
$nN$. We also write
\begin{equation}\label{sample RN}
R_N(V,\beta)=(\beta-\gamma_1\otimes C_N-\gamma_2 \otimes VD_NV^*)^{-1},
\quad\beta\in\mathcal G_N,V\in{\mathrm U}(N), 
\end{equation}
for sample values of this random variable. The function $\omega_N$ is given by
\begin{equation}\label{omegaN}
\omega_N(\beta)=\gamma_1\otimes C_N
+(\mathbb E(R_N(\beta)))^{-1},\quad\beta\in\mathcal G_N. 
\end{equation}

We start by showing that the matrix $\omega_N(\beta)$
has a block diagonal form, thus extending \cite[Lemma 4.7]{BBCF}.
We recall that the commutant and double commutant of a set $S\subset M_m (\mathbb C)$
are denoted by $S'$ and $S''$, respectively.
We use the fact that  $M_n(\mathbb C)\otimes S''=(I_n\otimes S')'$. If $S=\{C_N\}$ then $\{C_N\}''$ is the 
linear span of the matrices $\{I_N,C_N,\dots,C_N^{N-1}\}$.  In particular, every eigenvector of $C_N$ is a 
common eigenvector for the elements of $\{C_N\}''$.

For each $N\in\mathbb N$, we select an eigenbasis $\{f_N^{(1)},\dots,f_N^{(N)}\}$ for the operator $C_N$ 
and denote by $\lambda_N^{(j)}$ the corresponding eigenvalues, that is,
$C_Nf_N^{(j)}=\lambda_N^{(j)}f_N^{(j)}$. We write 
$P_N^{(j)}=f_N^{(j)}\otimes f_N^{(j)*}\in M_N(\mathbb C)$ for the orthogonal projection onto the space 
generated by $f_N^{(j)}$, $j=1,\dots,N$. Thus, the double commutant $\{C_N\}''$ is contained in the 
linear span of $\{P_N^{(j)}:j=1,\dots,n\}$.

We write ${\textbf{[}}x,y{\textbf{]}}=xy-yx$ for the commutator of two elements $x,y$ in an algebra.

\begin{lemma}\label{L2.1}
For every $\beta\in \mathcal G_N$ we have:  
\begin{itemize}
\item[(1)] $\mathbb E(R_{N}(\beta))\in M_n(\mathbb C)\otimes\{C_N\}''$. In particular, there exist analytic 
functions $\omega_N^{(j)}:\mathcal G_N\to M_n(\mathbb C)$, $j=1,\dots,N$, such that
$$
\omega_N(\beta)=\sum_{j=1}^N\omega_N^{(j)}(\beta)\otimes P_N^{(j)},\quad \beta\in\mathcal G _N.
$$
\item[(2)] For every $Z\in M_N(\mathbb C)$,
$$
\text{\bf[}\mathbb E(R_{N}(\beta)),I_n\otimes Z\text{\bf]}=
\mathbb E(R_{N}(\beta)\text{\bf[}\gamma_1\otimes C_N ,I_n\otimes Z\text{\bf]}R_{N}(\beta)).
$$
\end{itemize}
\end{lemma}

\begin{proof}
The first assertion in (1) follows from an application of (2) to an arbitrary matrix $Z\in\{C_N\}'$ and from 
the fact that $M_n(\mathbb C)\otimes\{C_N\}''=(I_n\otimes\{C_N\}')'$. The second assertion follows 
because we also have $\gamma_1\otimes C_N\in M_n(\mathbb C)\otimes\{C_N\}''$. To prove (2), observe 
that the analytic map 
$$
H(W)=\mathbb E
\left((\beta\otimes I_N-(\gamma_1\otimes c_N+\gamma_2\otimes e^{iW}d_Ne^{-iW}))^{-1}\right)
$$  
is defined for $W$ in an open neighbourhood  of the set of selfadjoint matrices in $M_N(\mathbb C)$. 
The unitary invariance of $d_N$ implies that $H(W)=\mathbb E(R(\beta))$ if $W$ is selfadjoint. Since
the selfadjoint matrices form a uniqueness set for analytic functions, we conclude that $H$ is constant
on an open subset of $M_N(\mathbb C)$ containing the selfadjoint matrices. Given an arbitrary $Z\in 
M_N(\mathbb C)$, we conclude that the function $\varepsilon\mapsto H(\varepsilon Z)$ is defined and 
constant for $\varepsilon\in\mathbb C$ with $|\varepsilon|$ sufficiently small. Differentiate with 
respect to $\varepsilon$ and set $\varepsilon=0$, we obtain
$$
\mathbb E(R_{N}(\beta)\text{\bf[}I_n\otimes Z,(\gamma_2\otimes d_N)\text{\bf]}R_{N}(\beta))=0.
$$
The equality $$R_{N}(\beta)(\gamma_2\otimes d_N)
=R_{N}(\beta)(\beta\otimes I_N-
\gamma_1\otimes C_N)-I_n\otimes I_N,$$ applied in the relation above, yields (2).
\end{proof}

The following result is simply a reformulation of Lemma \ref{L2.1} that emphasizes the fact that the 
functions $\beta\mapsto(\omega_N^{(j)}(\beta)-\lambda_N^{(j)}\gamma_1)^{-1}$ extend holomorphically 
to the open set $\{\beta\in\mathbb M_N(\mathbb C):\beta-\gamma_1\otimes c_N-\gamma_2\otimes 
d_N\text{ is invertible}\}.$

\begin{cor}\label{RN in terms of omegaN}
	We have
	$$
	(I_n\otimes P_N^{(j)})\mathbb E(R_N(\beta))(I_n\otimes P_N^{(j)})=
	(\omega_N^{(j)}(\beta)-\lambda_N^{(j)}\gamma_1)^{-1}\otimes P_N^{(j)},\quad j=1,\dots,N,
	$$
	for every $\beta\in\mathcal G_N$ such that 
	$\omega_N{(j)}(\beta)-\lambda_N^{(j)}\gamma_1$ is invertible.
\end{cor}

 It is useful to rewrite assertion (2) of Lemma \ref{L2.1} as follows:
\begin{align}
\text{\bf[}\omega_N (\beta),I_n\otimes Z\text{\bf]}
& = (\mathbb E(R_N(\beta)))^{-1}\mathbb E\left((R_N(\beta)-\mathbb E(R_N(\beta)))\frac{}{}\right.
\nonumber\\
&\left.\frac{}{}\times(\gamma_1\otimes\text{\bf[}Z,C_N\text{\bf]})
(R_N(\beta)-\mathbb E(R_N(\beta)))\right)(\mathbb E(R_N(\beta)))^{-1}.\label{New4.10}
\end{align}
This is {analogous to}
\cite[(4.10)]{BBCF} and the derivation is practically identical. Relation (\ref{New4.10}) allows us to estimate 
the differences between the matrices $\omega_N^{(j)}(\beta)$ once we control the differences 
$R_N(\beta)-\mathbb E(R_N(\beta)) $. For this purpose, we use the concentration of measure result
in \cite[Corollary 4.4.28]{AGZ}. This requires estimating the Lipschitz constant of the map $V\mapsto
R_N(V,\beta)$ (see \eqref{sample RN}) in the Hilbert-Schmidt norm.  We use the notation $\|T\|_2=
\textrm{Tr}_m(T^*T)$ for the Hilbert-Schmidt norm of an arbitrary matrix $T\in M_m(\mathbb C)$.

\begin{lemma}\label{Lip estimate}
	Suppose that $N\in\mathbb N$ and $\beta\in \mathcal G_N$. Then
	$$
	\|R_N(V,\beta)-R_N(W,\beta)\|_2\le 
	2r^2\|\gamma_2\|_2\|D_N\|\|V-W\|_2,
	\quad V,W\in {\mathrm U}(N),
	$$
	where $r=\|R_N(\beta)\|_{M_n(\mathbb C)\otimes\mathcal M_N}$.
\end{lemma}

\begin{proof}
A simple calculation shows that
$$
R_N(V,\beta)-R_N(W,\beta)=R_N(V,\beta)[\gamma_2\otimes(WD_NW^*-VD_NV^*)]R_N(W,\beta).
$$
Next we see that
$$
WD_NW^*-VD_NV^*=(W-V)D_NW^*+VD_N(W^*-V^*),
$$
and thus
$$
\|WD_NW^*-VD_NV^*\|_2\le 2\|D_N\|\|W-V\|_2.
$$
Use now the equality $\|T\otimes S\|_2=\|T\|_2\|S\|_2$ to deduce that 
\begin{align*}
\|R_N(V,\beta)&-R_N(W,\beta)\|_2\\
&\le
\|R_N(V,\beta)\|\|\gamma_2\|_2\|WD_NW^*-VD_NV^*\|_2 \|R_N(W,\beta)\|\\
&\le 2\|\gamma_2\|_2\|D_N\|\|R_N(V,\beta)\|
\|R_N(W,\beta)\|\|V-W\|_2.
\end{align*} 
The lemma follows from this estimate.
\end{proof}

\begin{prop}\label{prop.6.2}
	Suppose that $\sup\{\|D_N\|:N\in\mathbb N\}<+\infty$, $\beta\in \bigcap_{N\in\mathbb N}
\mathcal G_N$, and moreover,
	$$r=\sup\{\|R_N(\beta)\|_{M_n(\mathbb C)\otimes\mathcal M_N}:
	N\in\mathbb N\}<+\infty.
	$$
	Let $X_N,Y_N\in M_n\otimes M_N$ be matrices of norm one and rank uniformly bounded by 
$m\in\mathbb{N}$. Then:
	\begin{itemize}
		\item[(1)] Almost surely,
		\begin{equation}\label{concentration}
		\lim_{N\to\infty}\|X_N(R_N(\beta)-\mathbb E(R_N(\beta)))Y_N\|=0.
		\end{equation}
		\item[(2)] There exists $k>0$ such that
		\begin{equation}
		\mathbb{E}(\|X_N(R_N(\beta)-\mathbb E(R_N(\beta)))Y_N\|^2)\le\frac{k}{N}r^4, \quad 
N\in\mathbb{N}.\label{variance}
		\end{equation}
	\end{itemize}  
In particular, there exists a dense countable subset $\Lambda\subset\bigcap_{N\in\mathbb N}
\mathcal G_N$ such that almost surely, 
\eqref{concentration} holds for any $\beta\in\Lambda$.
\end{prop}

\begin{proof}
An arbitrary operator of rank $m$ can be written as a sum of $m$ operators of rank one (with the same or 
smaller norm). Thus we may, and do, restrict ourselves to the case in which the operators $X_N $ and 
$Y_N$ are projections of rank $1$. In this case, $X_NR_N(V,\beta)Y_N$ is a scalar multiple of a fixed 
operator of rank one and Lemma \ref{Lip estimate} shows that this function satisfies a Lipschitz 
estimate. This estimate, combined with \cite[Corollary 4.4.28]{AGZ}, shows that
$$
\mathbb P(\|X_N(R_N(\beta)-\mathbb E(R_N(\beta)))Y_N\|>{\varepsilon})
\le 2\exp\Big(\frac{-N\varepsilon^2}{8r^4\|\gamma_2 \|^2\|D_N\|^2}\Big)
$$
for every $\varepsilon>0$ and every $\alpha\in(0,1/2)$. The hypothesis implies that the last denominator 
has a bound independent of $N$. Part (1) of the lemma follows from this inequality, while (2) follows from 
the formula $\mathbb E(|Z|)=\int_0 ^{+\infty}\mathbb P (|Z|>t)\,dt$, valid for arbitrary random 
variables $Z$.
\end{proof}

\begin{rem}\label{rmk6.3} While Proposition \ref{prop.6.2} was formulated for $\beta\in
\bigcap_{N\in\mathbb N}\mathcal G_N$, the  hypothesis $r<+\infty$ (and therefore the conclusion of the 
proposition) is also satisfied  in the following cases:
\begin{itemize}
\item[(1)] $\Im \beta>0$ {with} $r\le\|(\Im \beta)^{-1}\|$;
\item[(2)] $\beta=ze_{1,1}-\gamma_0$ with $z\in\mathbb C^+\cup\mathbb C^-$, by an estimate 
provided by Lemma \ref{alta lema};
\item[(3)] $\beta=xe_{1,1}-\gamma_0$ with $x\in\mathbb R$, by the same estimate provided that 
$$
|x|>\sup\{\|P(c_N,d_N)\|_{M_n(\mathbb C)\otimes\mathcal M_N}\colon N\in\mathbb N\}.
$$
\end{itemize}
\end{rem}

\begin{cor}\label{cor7.4} Under the assumptions of Proposition {\em \ref{prop.6.2}} suppose that we also 
have $\sup\{\|C_N\|:N\in\mathbb N\}<+\infty,$ and set 
$$
t=\sup\{\|P(c_N,d_N)\|_{M_n(\mathbb C)\otimes\mathcal M_N}\colon N\in\mathbb N\}.
$$
Let $K\subset\mathbb C\setminus [-t,t]$ be a compact set. Then, almost surely, the functions
$$
z\mapsto\|X_N(R_N(ze_{1,1}-\gamma_0)-\mathbb E(R_N(ze_{1,1}-\gamma_0)))Y_N\|
$$
converge to zero uniformly on $K$ as $N\to+\infty$.
\end{cor}

\begin{proof}
According to Proposition \ref{prop.6.2} and Remark \ref{rmk6.3}(2) and (3), almost surely, for every 
$z\in\mathbb C\setminus[-t,t]$ such that $\Im z\in\mathbb{Q}$ and $\Re z\in \mathbb{Q}$, we have 
pointwise convergence to zero.  A second application of Remark \ref{rmk6.3} yields uniform bounds on 
$K$ for all of these functions, and this implies uniform convergence because the resolvents involved are 
analytic in $z$.
\end{proof}

We apply the concentration results just proved to operators $X_N$ and $Y_N$, of the form 
$I_n\otimes P_N^{(j)}$, where $\{P_N^{(j)}:j=1,\dots,N\}$ are the projections used in Lemma \ref{L2.1}. 
The rank of $I_n\otimes P_N^{(j)}$ is equal to $n$.

\begin{prop}\label{P}
	Suppose that 
	$\sup\{\|C_N\|+\|D_N\|:N\in\mathbb N\}<+\infty,$ and let
$\beta\in\bigcap_{N\in\mathbb N}\mathcal G_N$ be such that
$$
r=\sup\{\|R_N(\beta)\|_{M_n(\mathbb C)\otimes\mathcal M_N}\colon N\in\mathbb N\}<+\infty
$$
and
$$
r'=\sup\{\|(\mathbb E(R_N(\beta)))^{-1}\|\colon N\in\mathbb N\}<+\infty.
$$
Then $\sup_{N\in\mathbb N}N\|\omega_N(\beta)-\omega_N^{(1)}(\beta)\otimes I_N\|<+\infty $.
\end{prop}
\begin{proof} The conclusion of the propostion is equivalent to
	$$\max_{j\in\{2,\dots,N\}}\|\omega_N^{(j)}(\beta)-\omega_N^{(1)}(\beta )\|=O(1/N)$$
 	as $N\to+\infty$.
Fix $j\in\{2,\dots,N\}$ and set $Z_N=f_j\otimes f_1^*\in M_N(\mathbb C^N)$. Thus, $Z_N$ is an 
operator of rank one such that $Z_Nh=\langle h,f_1\rangle f_j$ for every $h\in\mathbb C^N$. We have
\begin{align*}
\text{\bf[}\omega_N(\beta),I_n\otimes Z_N\text{\bf]}&=
(\omega_N^{(j)}(\beta)-\omega_N^{(1)}(\beta))\otimes Z_N\\
&=(I_n\otimes P_N^{(j)})((\omega_N^{(j)}(\beta)-\omega_N^{(1)}(\beta))\otimes Z_N)(I_n\otimes P_N^{(1)})
\end{align*}
and, similarly,
$$
\text{\bf[}C_N, Z_N\text{\bf]}=(\lambda_N^{(j)}-\lambda_N^{(1)})Z_N
= P_N^{(j)}(\lambda_N^{(j)}-\lambda_N^{(1)})Z_N P_N^{(1)}.
$$
Next, we apply (\ref{New4.10}) and use the fact that $I\otimes P_N^{(j)}$ commutes with $\mathbb E 
(R_N(U_N,\beta))$. Setting $X_N=I_n\otimes P_N^{(k)}$, $Y_N=I_n\otimes P_N^{(1)}$, we obtain
\begin{align*}
(\omega_N^{(j)}(\beta)&-\omega_N^{(1)}(\beta))\otimes Z_N=
(\lambda_N^{(j)}-\lambda_N^{(1)})(\mathbb E (R_N(\beta)))^{-1}\\
&\times \mathbb E(X_N(R_N(\beta)-\mathbb E(R_N(\beta)))
X_N(\gamma_1\otimes Z_N)Y_N
(R_N(\beta)-\mathbb E(R_N(\beta)))Y_N)\\
&\times (\mathbb E (R_N(\beta)))^{-1}.
\end{align*}
Since $\|Z_N\|=1$, an application of the Cauchy-Schwarz inequality leads to the following estimate:
\begin{align*}
\|\omega_N^{(j)}(\beta)&-\omega_N^{(1)}(\beta)\|\le|\lambda_N^{(j)}-\lambda_N^{(1)}|
\|(\mathbb E (R_N(\beta)))^{-1}\|^2\|\gamma_1\|\\
&\times \mathbb E(\|X_N(R_N-\mathbb E(R_N(\beta)))X_N\|^2)^{1/2}
\mathbb E(\|Y_N(R_N-\mathbb E(R_N(\beta)))Y_N\|^2)^{1/2}.
\end{align*}
We have $|\lambda_N^{(j)}-\lambda_N^{(1)}|\le2\|C_N\|$ and the product of the last two factors above 
is estimated via (\ref{variance}) by $kr^4/N$ with $k$ independent of $N$. Thus,
$$
\|\omega_N^{(j)}(\beta)-\omega_N^{(1)}(\beta)\|\le 2k\|C_N\|r^{\prime2}\|\gamma_1\|r^4/N\le k'/N,
$$
with $k'$ independent of $j$ and $N$. The lemma follows.
\end{proof}

In the probability model we consider, the sequences $\{C_N\}_{N\in\mathbb N}$ and 
$\{D_N\}_{N\in\mathbb N}$ are uniformly bounded in norm and, in addition, the sequences 
$\{\mu_{C_N}\}_{N\in\mathbb N}$ and $\{\mu_{D_N}\}_{N\in\mathbb N}$ converge weakly to $\mu$ and 
$\nu$, respectively. We denote by $(a,b)$ a pair of free random variables in some $C^*$-probability space
$(\mathcal M,\tau)$ such that the sequence $\{(c_N,d_N)\}_{N\in\mathbb N}$ converges in distribution to 
$(a,b)$. We also set
 $$
\mathcal G=\mathcal G_{\gamma_1\otimes a,\gamma_2\otimes b,M_n(\mathbb C),M_n(\mathbb 
C\langle a\rangle)},\quad\omega=\omega_{\gamma_1\otimes a,\gamma_2\otimes b,M_n(\mathbb 
C),M_n(\mathbb C\langle a\rangle)}.
$$
In other words, $\omega$ is the usual matrix subordination function associated to the pair $(\gamma_1\otimes a,\gamma_2\otimes b)$.

 \begin{prop}\label{omega N vs omega}
 	Suppose that $\sup\{\|C_N\|+\|D_N\|:N\in\mathbb N\}<+\infty$ and that the sequence 
$\{(c_N,d_N)\}_{N\in\mathbb N}$ converges to $(a,b)$ in distribution. Let $\mathcal D\subset
\mathcal G\cap\bigcap_{n\in\mathbb N}\mathcal G_N$ be a connected open set containing $\mathbb 
H^+(M_n(\mathbb C))$ and such that the sequence of functions $\{\|\omega_N\|\}_{N\in\mathbb N}$ is 
locally uniformly bounded on $\mathcal D$. Then 
$$
\lim_{N\to\infty}\|\omega_N(\beta)-\omega(\beta)\otimes I_N\|=0,\quad\beta\in\mathcal D.
$$
\end{prop}

\begin{proof} By hypothesis, the analytic functions $\{\omega_N^{(1)}\}_{N\in\mathbb N}$ form a normal 
family on $\mathcal D$. By Proposition \ref{P}, it suffices to prove that every subsequential limit of this 
sequence equals $\omega$. Suppose that $\{\omega_{N_k}^{(1)}\}_{k\in\mathbb N}$ converges on 
$\mathcal D$ to a function $\widetilde{\omega}$. Fix $\beta\in H^+(M_n(\mathbb C))$ such that 
$\Im\beta>\sup\{\|C_N\|+\|D_N\|:N\in\mathbb N\}$.  Then
 	\begin{align*}
 \mathbb E(R_N(\beta))&=( \omega_N^{(1)}(\beta)\otimes I_N-\gamma_1\otimes C_N+O(1/N))^{-1}\\
 	&=( \omega_N^{(1)}(\beta)\otimes I_N-\gamma_1\otimes C_N)^{-1}+O(1/N)\\
 	&=\sum_{m=0}^\infty (\omega_N^{(1)}(\beta)^{-1})(\gamma_1
 	\omega_N^{(1)}(\beta)^{-1})^m\otimes C_N^m+O(1/N),
 	\end{align*}
 	and thus
 	$$
 	(\mathrm{Id}_{M_n(\mathbb C)}\otimes\mathrm{tr}_N)(\mathbb E(R_N(\beta)))
 	=\sum_{m=0}^\infty(\omega_N^{(1)}(\beta)^{-1})(\gamma_1
 	\omega_N^{(1)}(\beta)^{-1})^m\cdot\mathrm{tr}_N(C_N^m)+O(1/N).
 	$$
 	Setting $N=N_k$, letting $k\to+\infty$, and observing that the series on the right
 	is uniformly dominated, we conclude that
 	\begin{align*} 	
 	\lim_{k\to\infty}(\mathrm{Id}_{M_n(\mathbb C)}\otimes\mathrm{tr}_{N_k})
 	(\mathbb E(R_{N_k}(\beta)))&=
 	\sum_{m=0}^\infty(\widetilde{\omega}(\beta)^{-1})(\gamma_1
 	\widetilde{\omega}(\beta)^{-1})^m\cdot\tau(a^m)\\
 &=(\mathrm{Id}_{M_n(\mathbb C)}\otimes\tau)((\widetilde{\omega}(\beta)-\gamma_1\otimes a)^{-1}).
 	\end{align*}
 	The fact that the pairs $(C_N,D_N)$ converge to $(a,b)$ implies that
 	\begin{align*}
 	\lim_{k\to\infty}(\mathrm{Id}_{M_n(\mathbb C)}\otimes\mathrm{tr}_{N_k})
 	(\mathbb E(R_{N_k}(\beta)))&=
 (\mathrm{Id}_{M_n(\mathbb C)}\otimes\tau)((\beta-\gamma_1\otimes a-\gamma_2\otimes b)^{-1})\\
 	&=(\mathrm{Id}_{M_n(\mathbb C)}\otimes\tau)((\omega(\beta)-\gamma_1\otimes a)^{-1}) 
 	\end{align*}
 	Thus, we obtain the equality
 	$$
 	(\mathrm{Id}_{M_n(\mathbb C)}\otimes\tau)((\omega(\beta)-\gamma_1\otimes a)^{-1})=
 	(\mathrm{Id}_{M_n(\mathbb C)}\otimes\tau)((\widetilde{\omega}(\beta)-\gamma_1\otimes a)^{-1}),
 	$$
 	and this easily yields $\widetilde{\omega}(\beta)=\omega(\beta)$. Since $\mathcal D$ is connected, 
we must have $\widetilde{\omega}=\omega$, thus concluding the proof.
 	 \end{proof}
 	 
Corollary  \ref{cor:local bounds for omega} implies the following result.

\begin{cor}\label{cor:omega N vs omega z-gamma c}
Suppose that $\mathfrak s:=\sup\{\|C_N\|+\|D_N\|:N\in\mathbb N\}<+\infty$ and the sequence 
$\{(c_N,d_N)\}_{N\in\mathbb N}$ converges to $(a,b)$ in distribution. Let $\mathfrak k=\max\{k,
\mathfrak s\},$ where $k$ is the constant provided by 
Lemma {\em \ref{large |z|}}. Then, for every $z\in\mathbb C\setminus[-\mathfrak k,\mathfrak k]$, we 
have $ze_{1,1}-\gamma_0\in\mathcal G\cap\bigcap_{N\in\mathbb N}\mathcal G_N$ and
$$
\lim_{N\to\infty}\|\omega_N(ze_{1,1}-\gamma_0)-\omega(ze_{1,1}-\gamma_0)\otimes I_N\|=0.
$$
\end{cor} 	 

The preceding results combine to yield convergence results for sample resolvents. For the following 
statement, it is useful to identify $\mathbb C^m$ with a subspace of $\mathbb C^N$ if $m<N$ and to 
denote by $\{f_1,\dots,f_N\}$ the standard basis in $\mathbb C ^N$. Thus, we have 
$f_j\in\mathbb C^m$ provided that $j\le m$.

\begin{prop}\label{sample resolvents,ordinary convergence}
	 Suppose that $\sup\{\|C_N\|+\|D_N\|:N\in\mathbb N\}<+\infty$ and that the sequence 
$\{(c_N,d_N)\}_{N\in\mathbb N}$ converges to $(a,b)$ in distribution.  Suppose also that $C_N$ is 
diagonal in the standard basis, that is, $C_Nf_j=\lambda_N^{(j)}f_j$, $j=1,\dots N$, and that 
$p\in\mathbb N$ is such that the limits $\lambda^{(j)}=\lim_{N\to\infty}\lambda_N^{(j)}$ exist for $j=1,
\dots,p$. Denote by $P_N:\mathbb C^N\to\mathbb C^p$ the orthogonal projection onto $\mathbb C^p$, 
$N\ge p$. Let $\mathcal D\subset\mathcal G\cap\bigcap_{N\in\mathbb N}\mathcal G_N$ be a connected 
open set containing $\mathbb H^+(M_n(\mathbb C))$ and such that the sequence of functions 
$\{\|\omega_N\|\}_{N\in\mathbb N}$ is locally uniformly bounded on $\mathcal D$. Then almost surely
	 \begin{equation*}
	 \lim_{N\to\infty}(I_n\otimes P_N)R_N(\beta)(I_n\otimes P_N)=
	 (\omega(\beta)\otimes I_p-\gamma_1\otimes\text{\rm diag}(\lambda^{(1)},\dots,\lambda^{(p)}))^{-1}
	 \end{equation*} 
	 in the norm topology for every $\beta\in\mathcal D$.	
\end{prop}

\begin{proof}
By Proposition \ref{prop.6.2}, it suffices to show that the conclusion holds with $\mathbb E(R_N(\beta))$ 
in place of $R_N(\beta)$. We observe that
	 $$
	 (I_n\otimes P_N)\mathbb E(R_N(\beta))(I_n\otimes P_N)
	 =\sum_{j=1}^p(\omega_N^{(j)}(\beta)-\lambda_N^{(j)}\gamma_1)^{-1}\otimes P_N^{(j)},
	 $$
	 so the desired conclusion follows from Proposition \ref{omega N vs omega}.
\end{proof}

We observe for use in the following result that there exists a domain $\mathcal D$ as in the above 
statement such that $ze_{1,1}-\gamma_0\in\mathcal D$ for every $z\in\mathbb C^+$.

When the convergence of $(C_N,U_ND_NU_N^*)$ to $(a,b)$ is strong, the preceding result extends beyond
$\mathbb H^+(M_n(\mathbb C))$. In this case, $\sigma(P(C_N,U_ND_NU_N^*))$ converges almost surely 
to $\sigma(P(a,b))$ and thus the sample resolvent $R_N(U_N,ze_{1,1}-\gamma_0)$ is defined almost 
surely for large $N$ even for $z\in \mathbb R\setminus\sigma(P(a,b))$. We also recall that the function 
$$
(\omega(ze_{1,1}-\gamma_0)-\lambda\gamma_1)^{-1}
$$ 
extends analytically to $\mathbb C\setminus\sigma(P(a,b))$ if $\lambda\in\sigma(a)$.  These analytic 
extensions are used in the following statement.
\begin{prop}\label{sample resolvents under strong convergence}
	Under the hypothesis of Proposition {\em\ref{sample resolvents,ordinary convergence}}, suppose
 that the pairs $\{(C_N,U_ND_NU_N^*)\}_{N\in\mathbb N}$ converge strongly to $(a,b)$. 
Then almost surely
 \begin{equation*}
 \lim_{N\to\infty}(I_n\otimes P_N)R_N(\beta)(I_n\otimes P_N)=
 (\omega(\beta)\otimes I_p-\gamma_1\otimes\text{\rm diag}(\lambda^{(1)},\dots,\lambda^{(p)}))^{-1}
 \end{equation*}
 for $\beta=ze_{1,1}-\gamma_0$, $z\in\mathbb C\setminus\sigma(a,b)$. The convergence is 
uniform on compact subsets of $\mathbb C\setminus\sigma(P(a,b))$.
\end{prop}

\begin{proof}
	Strong convergence implies that $\lambda^{(j)}\in\sigma(a)$ for $j=1,\dots,p$, so the functions 
$(\omega(ze_{1,1}-\gamma_0)-\lambda^{(j)}\gamma_1)^{-1}$ extend analytically to $\mathbb R
\setminus\sigma(P(a,b))$. Let $\mathcal O\subset\mathbb C\setminus\sigma(P(a,b))$ be a connected 
open set containing $\{z\}\cup\mathbb C^+$ which is at a strictly positive distance from $\mathbb C
\setminus\sigma(P(a,b))$. We prove that the conclusion of the proposition holds for 
$U_N(\xi)$ provided that $\xi\in\Omega$ is such that the conclusion of Corollary \ref{cor7.4} holds and 
$\sigma(P(C_N,VD_NV^*))\subset\mathbb R\setminus\mathcal O$ for $V=U_N(\xi)$ and sufficiently large 
$N$. By hypothesis, the collection of such points $\xi\in\Omega$ has probability $1$. Lemma \ref{alta 
lema} shows that the family of functions $\|R_N(U_N(\xi),ze_{1,1}-\gamma_0)\|$ is locally uniformly 
bounded on $\mathcal O$ for large $N$. By Montel's theorem, we can conclude the proof by verifying the 
conclusions of the proposition for $\beta=ze_{1,1}-\gamma_0$ with $z\in \mathbb C^+$.  For such 
values of $z$, the result follows from Proposition \ref{sample resolvents,ordinary convergence}.		
	\end{proof}

\section{The unitarily invariant model}\label{sieben}

In this section we prove Theorem \ref{Main} under the additional condition that $A_N$ is a constant 
matrix and $B_N$ is a random unitary conjugate of another constant matrix. We may, and do, assume that 
for $N\ge p$, $A_N$ is diagonal in the standard basis $\{f_1,\dots,f_N\}$ with eigenvalues $\theta_1,
\dots,\theta_p,\lambda_N^{(p+1)},\dots\lambda_N^{(N)}$ and, as before, $B_N=U_ND_NU_N^*$. Then the 
random matrices $a_N=A_N\otimes1_\Omega$ and $d_N=U_ND_NU_N^*$ are viewed as elements of the 
noncommutative probability space $M_N(\mathbb C)\otimes L^\infty(\Omega)$. We fix free selfadjoint 
random variables $(a,b)$ in some tracial $W^*$-probability space such that the pairs $(a_N,d_N)$ 
converge in distribution to $(a,b)$ as $N\to\infty$.  This convergence is not strong because of the spikes 
$\theta_1,\dots,\theta_p$. As in \cite{BBCF}, we consider closely related pairs $(c_N,d_N)$ that do 
converge strongly to $(a,b)$. Namely, we set $c_N=C_N\otimes1_\Omega$, where $C_N$ is diagonal in 
the standard basis of $\mathbb C^N$ with eigenvalues $\lambda_N^{(1)},\dots,\lambda_N^{(N)}$ that
coincide with those of $A_N$ except that $\lambda_N^{(1)}=\cdots=\lambda_N^{(p)}=s$ is an arbitrary 
(but fixed for the remainder of this section) element of $\mathrm{supp}(\mu)$. For $N\ge p$, the 
difference $\Delta_N=A_N-C_N$ can then be written as $\Delta_N=P_N^*TP_N$, where $T\in 
M_p(\mathbb C)$ is the diagonal matrix with eigenvalues $\theta_1-s\dots,\theta_p-s$ and $P_N\colon
\mathbb C^N\to\mathbb C^p$ is the orthogonal projection.

According to Lemma \ref{o lema},  $\ker(tI_N-P(A_N,B_N))$ and $\ker(te_{1,1}\otimes I_N-L(A_N,B_N)))$ have 
the same dimension for every $t\in\mathbb R$. Setting $\beta=ze_{1,1}-\gamma_0$, strong convergence
of the pairs $c_N,d_N$ implies that almost surely the sample resolvent $R_N(U_N,\beta)$ is defined for 
sufficiently large $N$ if $z\in\mathbb C\setminus\sigma(P(c,d))$. (We continue using the notation 
introduced in (\ref{RN}), (\ref{sample RN}), and (\ref{omegaN}).) We need to consider the matrix
\begin{align*}
\beta\otimes I_N-\gamma_1\otimes& A_N-\gamma_2\otimes B_N=
\beta\otimes I_N-\gamma_1\otimes C_N-\gamma_2\otimes B_N-\gamma_1\otimes\Delta_N\\
&=(I_n \otimes I_N-(\gamma_1\otimes\Delta_N)R_N(U_N,\beta))(\beta\otimes I_N-\gamma_1\otimes 
C_N-\gamma_2\otimes B_N),
\end{align*}
since the order of $t\in\mathbb R$ as a zero of its determinant equals  $\dim\ker(tI_N-P(A_N,B_N))$, and 
hence the number of eigenvalues of $P(A_N,B_N)$ in a neighborhood $V$ of a given $t\in\mathbb R
\setminus\sigma(P(a,b))$ is the number of zeros of this determinant in $V$. Thus, we need to consider 
the zeros in $V$ of
$$
\det(I_n \otimes I_N-(\gamma_1\otimes\Delta_N)R_N(U_N,\beta)).
$$
Using Sylvester's identity ($\det(I_r-XY)=\det(I_p-YX)$ if $X$ is an $r\times p$ matrix and $Y$ is a 
$p\times r$ matrix) and the fact that $\Delta_N=P_N^*TP_N$, this determinant can be rewritten as
\begin{align*}
F_N(U_N,\beta)
=\det(I_n\otimes I_p-(\gamma_1\otimes T)(I_n\otimes P_N)R_N(U_N,\beta)(I_n\otimes P_N^*)).
\end{align*}
At this point, we observe that the hypothesis of Proposition \ref{sample resolvents under strong 
convergence} are satisfied with $\lambda^{(1)}=\cdots=\lambda^{(p)}=s$. We conclude that almost surely
$$
\lim_{n\to\infty}F_N(U_N,\beta)=F(\beta),\quad\beta=ze_{1,1}-\gamma_0,z\notin\sigma(P(a,b)),
$$
where
$$
F(\beta)=\det(I_n\otimes I_p-(\gamma_1(\omega(\beta)-s\gamma_1)^{-1})\otimes T),
$$
and the convergence is uniform on compact sets. The limit $F(ze_{1,1}-\gamma_0)$ is a (deterministic) 
analytic function on $\mathbb C\setminus\sigma(P(a,b))$.  An application of Hurwitz's theorem on zeros 
of analytic functions (see \cite[Theorem 5.2]{A}) yields the following result.
\begin{prop}\label{prop:almost main part 1}
	Suppose that $t_1,t_2\in\mathbb R$, $t_1<t_2$, $[t_1,t_2]\subset\mathbb R\setminus
\sigma(P(a,b))$, $F(t_je_{1,1}-\gamma_0)\ne0$ for $j=1,2$, and the function $F(ze_{1,1}-\gamma_0)$ 
has at most one zero $t$ in the interval $(t_1,t_2)$. Then, almost surely for large $N$, the matrix 
$P(A_N,B_N)$ has exactly $m$ eigenvalues in the interval $[t_1,t_2]$, where $m$ is the order of $t$ as a 
zero of $F(ze_{1,1}-\gamma_0)$ and $m=0$ if this function does not vanish on $[t_1,t_2]$.
\end{prop}

Part(1) of Theorem {\ref{Main}} is a reformulation of Proposition \ref{prop:almost main part 1}. To see that 
this is the case, we observe that $T$ is a diagonal matrix and thus the matrix
$$
I_n\otimes I_p-(\gamma_1(\omega(ze_{1,1}-\gamma_0)-s\gamma_1)^{-1})\otimes T
$$ 
is block diagonal with diagonal blocks 
$$
G_{j,s}(z)=I_n-(\theta_j\gamma_1-s\gamma_1)(\omega(ze_{1,1}-\gamma_0)-s\gamma_1)^{-1},\quad j=1,\dots,p.
$$
If $\det(G_{j,s}(z))$ has a zero of order $m_j$ at $t$ then the number $m$ in the statement is 
$m_1+\cdots+m_p$. We recall that $G_{j,s}$ is analytic   on $\mathbb C\setminus\sigma(P(a,b))$ but 
$\omega(ze_{1,1}-\gamma_0)$ is only meromorphic. It is not immediately apparent that the number 
$m_j$ does not depend on $s$ but this is a consequence of the following result.

\begin{lemma}\label{lem:independence on s}
	Suppose that $\alpha_1,\alpha_2\in M_n(\mathbb C)$ are such that  
	$(\omega(ze_{1,1}-\gamma_0)-\alpha_k)^{-1}$ extends analytically to $t$ for $k=1,2$. Then the 
order of $t$ as a zero of
	$$
	\det(I_n-(\theta_j\gamma_1-\alpha_k)(\omega(ze_{1,1}-\gamma_0)-\alpha_k)^{-1})
	$$
	does not depend on $k$.
\end{lemma}

\begin{proof}
	An easy calculation shows that
	$$
	I_n-(\theta_j\gamma_1-\alpha_k)(\omega(ze_{1,1}-\gamma_0)-\alpha_k)^{-1}=(\omega(ze_{1,1}-
	\gamma_0)-s\gamma_1)(\omega(ze_{1,1}-\gamma_0)-\alpha_k)^{-1}.
	$$
	The desired conclusion follows if we prove that the function
	$$
	H(z)=(\omega(ze_{1,1}-\gamma_0)-\alpha_1)(\omega(ze_{1,1}-\gamma_0)-\alpha_2)^{-1}
	$$
	is analytic and invertible at $z=t$. We have
	$$
	H(z)=I_n+(\alpha_2-\alpha_1)(\omega(ze_{1,1}-\gamma_0)-\alpha_2)^{-1}
	$$
	and
	$$
	H(z)^{-1}=I_n+(\alpha_1-\alpha_2)(\omega(ze_{1,1}-\gamma_0)-\alpha_1)^{-1},
	$$
	so the analyticity and invertibility of $H$ follow from the hypothesis.
\end{proof}

We proceed now to Part(2) of Theorem {\ref{Main}}. Thus, assumptions (A1--A3) and (B0--B2) are in force 
and, in addition, the spikes $\theta_1,\dots,\theta_p$ are distinct. In particular, 
$\sup\{\|C_N\|+\|D_N\|:N\in\mathbb N\}<+\infty$.

If $S\subset\mathbb R$ is a Borel set and $A\in M_N(\mathbb C)$ is a selfadjoint operator, then $E_A(S)$ 
denotes the orthogonal projection onto the linear span of the eigenvectors of $A$ corresponding to 
eigenvalues in $S$. For instance, under the hypotheses of Part(2) of Theorem {\ref{Main}}, $E_{A_N}
(\{\theta_j\})$ is a projection of rank one for $j=1,\dots,p$. If $h\colon\mathbb R\to\mathbb C$ is a 
continuous function, then $h(A)$ denotes the usual functional calculus for selfadjoint matrices.  Thus, if 
$Ax=tx$ for some $t\in\mathbb R$ and $x\in\mathbb C^N$, then $h(A)x=h(t)x$.

Fix $t\in\mathbb R\setminus\sigma(P(a,b))$ and $\varepsilon>0$ small enough. We need to show that, 
almost surely,
\begin{equation}\label{prj}
\lim_{N\to\infty}\mathrm{Tr}_N\left[E_{A_N}(\{\theta_i\})E_{P(A_N,B_N)}
((t-\varepsilon,t+\varepsilon))\right]=\delta_{i,i_0}{\mathcal C}_i(t).
\end{equation}
Choose $\delta>0$   and $N_0\in\mathbb N$ such that 
$[\theta_j-\delta,\theta_j+\delta]\cap\sigma(A_N)=\{\theta_j\}$, for $N\ge N_0$ and $ j=1,\dots,p$. 
Pick infinitely differentiable functions $f_j:\mathbb R\to[0,1]$ supported in $[\theta_j-\delta,
\theta_j+\delta]$ such that $f_j(\theta_j)=1$, $j=1,\dots,p$. Also pick an infinitely differentiable function 
$h:\mathbb R\to[0,1]$ supported in $(t-\varepsilon,t+\varepsilon)$ such that $h$ is identically $1$ on 
$[t-\varepsilon/2,t+\varepsilon/2]$. Then part (1) of Theorem {\ref{Main}} implies that, almost surely for 
large $N$, we have
$$
E_{A_N}(\{\theta_j\})E_{P(A_N,B_N)}((t-\varepsilon,t+\varepsilon))=f_j(A_N)h(P(A_N,B_N)).
$$
 
In anticipation of a concentration inequality, we prove a Lipschitz estimate for the functions $g_{N,j}\colon
{\mathrm U}(N)\to M_N(\mathbb C)$ defined by 
$$
g_{N,j}(V)=\mathrm{Tr}_N(f_j(A_N)h(P(A_N,VD_NV^*))),\quad V\in{\mathrm U}(N),j=1,\dots,p.
$$

\begin{lemma}There exists $k>0$, independent of $N$, such that
	$$
	|g_{N,j}(V)-g_{N,j}(W)|\le k\|V-W\|_2,\quad V,W\in{\mathrm U}(N), j=1,\dots,p.
	$$
\end{lemma}

\begin{proof} Given a Lipschitz function $u:{\mathrm U}(N)\to M_N(\mathbb C)$, we denote by 
$\mathrm{Lip}(u)$ the smallest constant $c$ such that
$$
\|u(V)-u(W)\|_2\le c\|V-W\|_2,\quad V,W\in{\mathrm U}(N),
$$
and we set $\|u\|_\infty=\sup\{\|u(V)\|:v\in{\mathrm U}(N)\}$. If $u_1,u_2:{\mathrm U}(N)\to 
M_N(\mathbb C)$ are two Lipschitz functions, then ${\mathrm{Lip}}(u_1+u_2)\le {\mathrm{Lip}}(u_2)
+{\mathrm{Lip}}(u_1)$ and
$$
{\mathrm{Lip}}(u_1u_2)\le \|u_1\|_\infty{\mathrm{Lip}}(u_2)+{\mathrm{Lip}}(u_1)\|u_2\|_\infty.
$$
Since the functions $V\mapsto V$ and $V\mapsto V^*$ are Lipschitz with constant $1$, we deduce 
immediately that the map $V\mapsto P(A_N,VD_NV^*)$ is Lipschitz with constant bounded independently 
of $N$. It is well-known that a Lipschitz function $f:\mathbb R\to\mathbb R$ is also Lipschitz, with the
same constant, when viewed as a map on the selfadjoint matrices with the Hilbert-Schmidt norm (see
for instance, \cite[Lemma A.2]{C}).  The function $h$ is infinitely differentiable with compact support, 
hence Lipschitz. We deduce that the map $V\mapsto h(P(A_N,VD_NV^*))$ is Lipschitz  with constant 
bounded independently of $N$. Finally, we have
	 $$
|g_{N,j}(V)-g_{N,j}(W)|\le \|f_j(A_N)\|_2\|h(P(A_N,VD_NV^*))-h(P(A_N,WD_NW^*))\|_2,
	 $$
and the lemma follows because $\|f_j(A_N)\|_2=1$.
\end{proof}

An application of \cite[Corollary 4.4.28]{AGZ} yields 
$$
\mathbb P\left(\left|g_{N,j}(U_N)-\mathbb E(g_{N,j}(U_N))\right|>\eta\right)\le
2\exp\left(-\frac{\eta^2N}{4k^2}\right),\quad\eta>0,j=1,\dots,p,
$$
and the Borel-Cantelli lemma shows that, almost surely,
\begin{equation}\label{BorelCantelli}
\lim_{N\to\infty}(g_{N,j}(U_N)-\mathbb E(g_{N,j}(U_N)))=0,\quad j=1,\dots,p.
\end{equation} 
The expected value  in (\ref{BorelCantelli}) is estimated  using \cite[Lemma 6.3]{C} and the fact that 
$f_j(A_N)$ is the projection onto the $j$th coordinate. If we set $r_N(z)=(z I_N-P(A_N,B_N))^{-1}$ and
$\widetilde R_N(ze_{1,1}-\gamma_0)=((ze_{1,1}-\gamma_0)\otimes I_N-L(A_N,B_N))^{-1}$, 
$z\in\mathbb C^+$, then
\begin{align}
\mathbb E(g_{N,j}(U_N))&= -\lim_{y\downarrow0}\frac1\pi\Im\int_\mathbb R\mathbb E(\mathrm{Tr}_N(
f_j(A_N)r_N(\xi+iy))h(\xi)\,\mathrm{d}\xi\nonumber\\
&  =-\lim_{y\downarrow0}\frac1\pi\Im\int_\mathbb R\mathbb E(r_N(\xi+iy)_{j,j})h(\xi)\,\mathrm{d}\xi.
\label{7.3}
\end{align}
The construction of the linearization $L$ (Section \ref{sec:linearization}) is such that the matrix 
$\widetilde R_N(ze_{1,1}-\gamma_0)$, viewed as an $n\times n$ block matrix, has $r_N(z)$ as its $(1,1)$ 
entry. By Propositions \ref{prop.6.2}(1) and \ref{sample resolvents,ordinary convergence} (with $A_N$ in 
place of $C_N$) and the unitary invariance of the distribution of $B_N$, for $z\in\mathbb C^+$ the matrices
$$
(I_n\otimes P_N)\mathbb E(\widetilde R_N(ze_{1,1}-\gamma_0))(I_n\otimes P_N)^*
$$ 
converge as $N\to\infty$ to the block diagonal matrix with diagonal entries
$$
(\omega(ze_{1,1}-\gamma_0)-\theta_j\gamma_1)^{-1}.
$$
It follows that
\begin{align}
\lim_{N\to\infty}
\mathbb E ((r_N(z))_{j,j})=((\omega(ze_{1,1}-\gamma_0)-\theta_j\gamma_1)^{-1})_{1,1}.\label{7.4}
\end{align}
We intend to let 
$N\to\infty$ in \eqref{7.3} using \eqref{7.4}, so we consider the differences 
$$
\Delta_{j,N}(z)=((\omega(ze_{1,1}-\gamma_0)-\theta_j\gamma_1)^{-1})_{1,1}
-\mathbb E (r_N(z)_{j,j}).
$$
By Corollary \ref{cor:local bounds for omega},
these functions are defined on $\mathbb C\setminus[-k,k]$ for some $k>0$ 
and satisfy $\Delta_{j,N}(\overline{z})=\overline{\Delta_{j,N}(z)}$. We claim that there exists
a sequence $\{v_N\}_{N\in\mathbb N}\subset(0,+\infty)$ such that $\lim_{N\to\infty}v_N=0$ and
\begin{equation}\label{estim}
|\Delta_{j,N}(z)|\le v_N\left(1+\frac{1}{(\Im z)^2}\right),\quad z\in\mathbb C^+.
\end{equation}
To verify this claim, we observe first \cite{akhieser} that the function $\mathbb E (r_N(z)_{j,j})$ is the Cauchy-Stieltjes
transform of a Borel probability measure $\sigma_{N,j}$ on $\mathbb R$.
Since
$$
\sup\{\|P(A_N,B_N)\|\colon N\in\mathbb N\}<\infty,
$$
the measures $\{\sigma_{N,j}\}_{N\in\mathbb N}$ have uniformly bounded supports. Now, \eqref{7.4} shows that the Cauchy-Stieltjes transform of any accumulation point of this sequence of measures is equal to $((\omega(ze_{1,1}-\beta_0)-\theta_j\beta_1)^{-1})_{1,1}$. It follows that this  sequence has
a weak limit $\sigma_j$ 
that is a Borel probability measure with compact support. The existence of the sequence 
$\{v_N\}_{N\in\mathbb N}$ follows from  \cite[Lemma 4.1]{BBCF} applied to the signed measures 
$\rho_N=\sigma_{N,i}-\sigma_i$.

We use \eqref{estim}, \eqref{7.3}, and the Lemma from \cite[Appendix]{CD} to obtain 
\begin{align}
\lim_{N\to\infty}\mathbb E(g_{N,j}(U_N))
& = -\lim_{N\to\infty}\lim_{y\downarrow0}\frac1\pi\Im\int_\mathbb R\mathbb E(r_N(\xi+iy)_{j,j})h(\xi)\,\mathrm{d}\xi\nonumber\\
& =  \lim_{N\to\infty}\lim_{y\downarrow0}\frac1\pi\Im\int_\mathbb R\Delta_{j,N}(\xi+iy)h(\xi)\,
\mathrm{d}\xi\nonumber\\
&  \quad-\lim_{y\downarrow0}\frac1\pi\Im\int_\mathbb R((\omega((\xi+iy)e_{1,1}-\gamma_0)-
\theta_j\gamma_1)^{-1})_{1,1}h(\xi)\,\mathrm{d}\xi\nonumber\\
& =  -\lim_{y\downarrow0}\frac1\pi\Im\int_\mathbb R((\omega((\xi+iy)e_{1,1}-\gamma_0)-
\theta_j\gamma_1)^{-1})_{1,1}h(\xi)\,\mathrm{d}\xi.
\label{7.6}
\end{align}
The choice of $h$, and the fact that 
$$
u(z)=((\omega(ze_{1,1}-\gamma_0)-\theta_j\gamma_1)^{-1})_{1,1}
$$ 
is analytic and real-valued on the intervals $[t-\varepsilon,t-\varepsilon/2]$ and $[t+\varepsilon/2,
t+\varepsilon]$, imply that the last line in \eqref{7.6} can be rewritten as
\begin{align}
 \lim_{y\downarrow0}&\frac1\pi\int_{t-\varepsilon}^{t-\frac\varepsilon2}\Im(u(\xi+iy))h(\xi)\,\mathrm{d}\xi+
 \lim_{y\downarrow0}\frac1\pi\int_{t+\frac\varepsilon2}^{t+\varepsilon}\Im(u(\xi+iy))h(\xi)\,\mathrm{d}\xi \nonumber\\
 &-\lim_{y\downarrow0}\frac1\pi\int_{t-\frac\varepsilon2}^{t+\frac\varepsilon2}
 \Im(u(\xi+iy))\,\mathrm{d}\xi\nonumber\\
& =  \lim_{y\downarrow0}\frac1{2\pi i}\int_{t-\frac\varepsilon2}^{t+\frac\varepsilon2}
u(\xi+iy)\,\mathrm{d}\xi -\lim_{y\downarrow0}\frac1{2\pi i}\int_{t-\frac\varepsilon2}^{t+\frac\varepsilon2}
u(\xi-iy)
\,\mathrm{d}\xi.\label{partial residue}
\end{align}
Recall (see, for instance, \cite[Chapter 4]{A}) that if $f$ is an analytic function on a simply connected 
domain $D$, except for an isolated singularity $a$, then $\frac1{2\pi i}\int_\gamma f(z)\,\mathrm{d}z
=n(\gamma,a)\mathrm{Res}_{z=a}f(z)$. Here $\gamma$ is a closed Jordan path in $D$ not containing 
$a$, $n(\gamma,a)$ is the winding number of $\gamma$ with respect to $a$, and $\mathrm{Res}_{z=a}
f(z)$ is that number $R$ which satisfies the condition that $f(z)-\frac{R}{z-a}$ has vanishing period
(called the residue of $f$ at $a$). Denote by $\Gamma_y$ the rectangle with corners $t\pm(\varepsilon/2)
\pm iy$ and let $\gamma_y$ be the boundary of $\Gamma_y$ oriented counterclockwise. The expression 
in \eqref{partial residue} represents the integral of $u$ on the horizontal segments in $\gamma_y$. It is 
clear that the integral of $u$ on the vertical segments is $O(y)$, and thus \eqref{7.6} implies the equality
$$
\lim_{N\to\infty}\mathbb E(g_{N,j}(U_N))=
\lim_{y\downarrow0}\frac1{2\pi i}\int_{\gamma_y}u(z)\,{\mathrm d}z=\mathrm{Res}_{z=t}u(z).
$$
The alternative formula in Theorem \ref{Main} follows from the fact that $u$ has a simple pole at $t$ 
because $u$ maps $\mathbb C^+$ to $\mathbb C^-$.

\section{The Wigner model}\label{cvFw}

We proceed now to the proof of Theorem
 \ref{thm8.2}.  The matrices  $A_N$
  are subject to the hypotheses (A1--A3), while $X_N/\sqrt{N}$ and
    $X_N$ satisfies conditions (X0--X3).   By \cite[Section 3, Assertion 2]{M}, it suffices to proceed under the
 additional hypothesis that each $A_N$ is a 
 constant matrix. The
     free variables $a$ and $b$ are such that
     $b$ has standard
semicircular  distribution $\nu_{0,1}$.

One consequence of the fact   that $b$ is a semicircular variable is that the function
$\omega$ is analytic on the entire set
$$
\{\beta\in M_n(\mathbb C):\beta\otimes 1-\gamma_1\otimes a-\gamma_2\otimes b\text{ is invertible.}\}
$$
 This
justifies the comment from Remark \ref{rmk5.2}. We recall that in this special case the subordination function is given by
$$\omega(\beta)
=\beta-\gamma_2(\mathrm Id_{M_n(\mathbb C)}
\otimes\tau)(\beta\otimes 1-\gamma_1\otimes a
-\gamma_2\otimes b)^{-1}\gamma_2.$$
Since the distribution of the random matrix $X_N$ is not usually invariant under unitary conjugation, we 
can no longer assume that $A_N$ is diagonal in the standard basis $\{f_1,\dots,f_N\}$. There is however a 
(constant) unitary matrix $V_N\in M_N(\mathbb C)$ such that $A_N$ is diagonal in the  basis $\{V_Nf_1,
\dots,V_Nf_N\}$ with eigenvalues $\theta_1,\dots,\theta_p,\lambda_N^{(p+1)},\dots\lambda_N^{(N)}$, 
$N\ge p$.  Viewing each realization of the random matrices $A_N$ and $X_N/\sqrt{N}$ as elements of the 
noncommutative probability space $\left(M_N(\mathbb C),\mathrm{tr}_N\right)$, almost surely the pairs 
$(A_N,X_N/\sqrt{N})$ converge in distribution, but not strongly, to $(a,b)$ as $N\to\infty$.  A 
modification of $A_N$ provides almost surely  strongly convergent pairs $(C_N,X_N/\sqrt{N})$.  Thus, let 
$C_N$ be diagonal in the  basis $\{V_Nf_1,\dots,V_Nf_N\}$ with eigenvalues $\lambda_N^{(1)},\dots,
\lambda_N^{(N)}$ that coincide with those of $A_N$ except that
$\lambda_N^{(1)}=\cdots=\lambda_N^{(p)}=s$ is an arbitrary (but fixed for the remainder of this section) element of $\mathrm{supp}(\mu)$.
For $N\ge p$, the difference $\Delta_N=A_N-C_N$ can then be written as $\Delta_N=V_NP_N^*TP_NV_N^*$, where $T\in M_p(\mathbb C)$ is the diagonal matrix with eigenvalues $\theta_1-s\dots,\theta_p-s$ and $P_N:\mathbb C^N\to\mathbb C^p$ is the orthogonal projection. Almost surely, the pairs $(C_N,X_N/\sqrt{N})$ converge strongly to $(a,b)$ as shown in 
\cite[Theorem 1.2]{BC} and \cite[Proposition 2.1]{CM}.  We continue using the notation introduced in (\ref{RN})  and (\ref{omegaN}) with $c_N=C_N$ and $d_N=\frac{X_N}{\sqrt{N}}$.
 The calculation in Section \ref{sieben} show that, almost surely, for $N$ large enough,  the number of eigenvalues of $P(A_N,X_N/\sqrt{N})$ in a small enough neighborhood of $t\in\mathbb R\setminus\sigma(P(a,b))$ is equal to the number of zeros of $F_N(ze_{1,1}-\gamma_0)$ in this neighborhood, where
$$
F_N(ze_{1,1}-\gamma_0)=
\det(I_n\otimes I_p-(\gamma_1\otimes T)(I_n\otimes P_NV_N^*)R_N(ze_{1,1}-\gamma_0)(I_n\otimes P_N^*V_N))
$$
is a random analytic function.
We focus on the study of the large $N$ behavior of the matrix function
\begin{equation}\label{GN}
\mathcal{F}_N(\beta)=(I_n\otimes P_NV_N^*)R_N(\beta)(I_n\otimes V_N P_N^*).
\end{equation}
We start with the special case in which $X_N$ is replaced by a standard G.U.E.. The following proposition is a consequence of the results in Section \ref{sieben}. 
Thus, suppose that $(X^{g}_N)_{N\in\mathbb N}$ is a sequence of standard G.U.E. ensembles
and we set 
${R}_N^g(\beta)=(\beta\otimes I_N-L(C_N,X^g_N/\sqrt{N}))^{-1},$
and
\begin{equation}
\label{GNW}
\mathcal{F}_N^g(\beta)=(I_n\otimes P_NV_N^*){R}_N^g(\beta)(I_n\otimes V_N P_N^*).
\end{equation}

\begin{prop}\label{gaussien}
We have $$\lim_{N\to\infty}\mathbb E( \mathcal {F}_N^g(\beta))=
(\omega(\beta)-s\gamma_1)^{-1}\otimes I_p$$
for every $\beta \in\mathbb H^+( M_n(\mathbb C))$.
\end{prop}
\begin{proof} Since G.U.E. ensembles are invariant under unitary conjugation we may, and do, assume that $V_N=I_N$ for every $N\in\mathbb N$. For every $\beta\in\mathbb H^+(M_N(\mathbb C))$ we have
	$$
\mathbb E( \mathcal{F}_N^g(\beta))=
\mathbb E( \mathcal{F}_N^g(\beta){\mathbf 1}_{\|X^g_N/\sqrt{N}\|\le 3})+
\mathbb E( \mathcal{F}_N^g(\beta)
{\mathbf 1}_{\|X^g_N/\sqrt{N}\|> 3}),	
	$$
	and the second term is at most
	$\|(\Im\beta)^{-1}\|\mathbb{P}(\|X^g_N/\sqrt{N}\|>3)$. As shown by Bai and Yin \cite{BYi}, this number tends to zero as $N\to\infty$. To estimate the first term, we recall that
	$X^g_N/\sqrt{N}=U_ND_NU_N^*$, where $U_N$ is a random matrix uniformly distributed in $\mathrm{U}(N)$ and $D_N$ is a random diagonal matrix, independent from $U_N$,
	whose empirical  spectral measure  converges almost surely to $\nu_{0,1}$ as $N\to\infty$.
	Thus, we can write\\

	 $
 \mathbb E( \mathcal{F}_N^g(\beta){\mathbf 1}_{\|X^g_N/\sqrt{N}\|\le 3})$ $$=\int_{\Omega} \int_{\Omega}
 (I_n\otimes P_N)R_N((\xi_1,\xi_2), \beta)(I_n\otimes P_N^*)\,\mathrm{d}\mathbb P(\xi_1)
 {\mathbf 1}_{\|D_N(\xi_2)\|\le 3}\,\mathrm{d}\mathbb P(\xi_2),
 $$
	$$R_N((\xi_1,\xi_2), \beta):=(\beta\otimes I_N-L( C_N, U_N(\xi_1)D_N(\xi_2)U_n(\xi_1)^*))^{-1}.$$
Proposition \ref{sample resolvents,ordinary convergence} can be applied for almost every $\xi_2$  and it shows that
$$
\int_{\Omega}
(I_n\otimes P_N)R_N((\xi_1,\xi_2),\beta)(I_n\otimes P_N^*)\,\mathrm{d}\mathbb P(\xi_1) {\mathbf 1}_{\|D_N(\xi_2)\|\le 3}
$$
converges to $(\omega(\beta)-s\gamma_1)^{-1}\otimes I_p$. The proposition follows now from an application of the dominated convergence theorem.
\end{proof}

Passing to arbitrary Wigner matrices requires
an approximation procedure from \cite[Section 2]{BC}. For every $\varepsilon >0$, there exist random selfadjoint matrices  
$X_N(\varepsilon)=[(X(\varepsilon))_{ij}]_{1\leq i,j\leq N}$ such that  
\begin{itemize}
\item[(H1)] the  variables $\sqrt{2}\Re  X_{ij}(\varepsilon)$, $\sqrt{2}\Im   X_{ij}(\varepsilon)$, 
 $X_{ii}(\varepsilon)$, $i,j\in \mathbb{N}$, $i<j$, are independent, 
centered with variance $1$ and satisfy a Poincar\'e inequality with common constant $C_{PI}(\epsilon)$,
\item[(H2)] for every $m\in\mathbb{N}$,
\begin{equation}\label{moments} 
\sup_{(i,j)\in\mathbb{N}^2}\mathbb{E}\left(\vert X_{ij}(\varepsilon)\vert^m\right)<+\infty,
\end{equation}
\end{itemize}
and almost surely for large $N$, 
$$
\left\|\frac{X_N -{X_N(\varepsilon)} }{\sqrt{N}}\right\|\leq\varepsilon.
$$
Set
\begin{equation}
R_N^\varepsilon (\beta)=
(\beta\otimes I_N-L(C_N,{X_N(\varepsilon)}/{\sqrt{N}}))^{-1}
\end{equation}
and
\begin{equation}\label{GNeps}
 \mathcal{F}^{\varepsilon}_N(\beta)=(I_n\otimes P_N V_N^*)
 R_N^\varepsilon (\beta)
 (I_n\otimes V_NP_N^*)
\end{equation}
for $\beta\in\mathbb H^+(M_n(\mathbb C^N))$. It readily follows that,  almost surely for large $N$, 
\begin{equation}\label{di} 
\left\|\mathcal{F}^{\varepsilon}_N(\beta)-\mathcal{F}_N(\beta)\right\|\leq{\varepsilon}\|\gamma_2\|\|(\Im \beta)^{-1}\|^2.
\end{equation}
Properties (H1) and (H2) imply that, for every $\varepsilon>0$, 
$$
\forall (i,j)\in\mathbb{N}^2,\;\kappa_1^{i,j,\varepsilon}=0,\;\kappa_2^{i,j,\varepsilon}=1,
$$
$$
\forall (i,j)\in\mathbb{N}^2,\;,i\neq j,\;\widetilde\kappa_1^{i,j,\varepsilon}=0,\;\widetilde\kappa_2^{i,j,\varepsilon}=1,
$$ 
and for any $m\in\mathbb{N}\setminus\{0\}$, 
\begin{equation}\label{cumulants} 
\sup_{(i,j)\in\mathbb{N}^2}|\kappa_m^{i,j,\varepsilon}|<+\infty,\;\sup_{(i,j)\in\mathbb{N}^2}
|\widetilde\kappa_m^{i,j,\varepsilon}|<+\infty,
\end{equation}
where for $i\neq j$, $(\kappa_m^{i,j,\varepsilon})_{m\geq1}$ and 
$(\widetilde \kappa_m^{i,j,\varepsilon})_{m\geq 1}$ denote  the classical 
cumulants of $\sqrt{2}\Re X_{ij}(\varepsilon)$ and $\sqrt{2}\Im  X_{ij}(\varepsilon)$ 
respectively, and $(\kappa_m^{i,i,\varepsilon})_{m\geq 1}$ denote the classical 
cumulants of $ X_{ii}(\varepsilon)$ (we set $(\widetilde\kappa_m^{i,i,\varepsilon})_{m\geq 1}\equiv 0$).

 We use the following notation for an arbitrary matrix $M\in M_n(\mathbb{C}) \otimes M_N(\mathbb{C})$:  
\begin{equation}\label{mij}
M^{ij}=(\mathrm{Id}_{M_n(\mathbb C)}\otimes\mathrm{Tr}_N)
\left(M\left(I_n\otimes\hat e_{j,i}\right)\right)\in 
M_n(\mathbb{C}),
\end{equation} 
and 
$$
M_{ij}=(\mathrm{Tr}_n\otimes\mathrm{Id}_{M_N(\mathbb C)})\left(M\left(e_{j,i}\otimes I_N\right)\right)\in 
M_N(\mathbb{C}),
$$ 
where $e_{j,i}$ (resp. $\hat e_{j,i}$) denotes the $n\times n$ (resp. $N\times N$) matrix whose unique 
nonzero entry equals 1 and occurs in row $j$ and column $i$.

\begin{prop}\label{comptilde} 

There exists a polynomial $P_\varepsilon$ in one variable with nonnegative coefficients such that for all large $N$, for  
every $v,u\in\{1,\ldots,N\}$, for every $\beta\in\mathbb{H}^+(M_n(\mathbb{C}))$, and for every deterministic 
$B_N^{(1)}, B_N^{(2)}\in M_n(\mathbb{C})\otimes M_N(\mathbb{C})$ such that $\|B_N^{(1)}\|\leq1$ 
and $\Vert B_N^{(2)}\Vert\leq 1$, we have 
\begin{equation}\label{gau} 
\left\|\mathbb{E}(B_N^{(1)} R_N^g(\beta)B_N^{(2)})^{vu}- 
\mathbb{E}( B_N^{(1)}  R^{{\tiny {\varepsilon}}}_N(\beta) B_N^{(2)})^{vu}\right\|\leq 
\frac{1}{\sqrt{N}}P_\varepsilon(\|(\Im \beta)^{-1}\|),
\end{equation}
and
\begin{equation}\label{comparaison}
\left\|\mathbb{E}(\mathcal{F}^g_N(\beta))-\mathbb{E}( \mathcal{F}_N^\varepsilon(\beta))\right\|\leq 
\frac{1}{\sqrt{N}}P_\varepsilon(\left\|(\Im w)^{-1} \right\|).
\end{equation}
\end{prop}

The proof uses a well-known lemma.

\begin{lemma}\label{IPP}
Let $Z$ be a real-valued random variable such that $\mathbb E(|Z|^{p+2})<\infty$. Let 
$\phi\colon\mathbb R\to\mathbb C$ be a function whose first $p+1$ derivatives are continuous and 
bounded. Then, 
\begin{equation}\label{IPP2}
\mathbb E(Z\phi(Z))=\sum_{a=0}^p\frac{\kappa_{a+1}}{a!}\mathbb E(\phi^{(a)}(Z))+\eta, 
\end{equation} 
where $\kappa_{a}$ are the cumulants of $Z $, $|\eta|\leq C\sup_t|\phi ^{(p+1)}(t)|
\mathbb E (|Z|^{p+2})$, and $C$ only depends on $p$.
\end{lemma}

\begin{proof}[Proof of Proposition {\em \ref{comptilde}}]

Following the approach of \cite[Ch. 18 and 19]{PS} we  introduce the interpolation matrix 
$X_\varepsilon(\alpha)=\cos\alpha X_N(\varepsilon)+\sin\alpha Y_N$,  $\alpha\in[0,{\pi}/{2}]$,
and the corresponding resolvent
$$ R_N^{\varepsilon,\alpha}(\beta)= 
( \beta\otimes I_N-L(C_N,{X_\varepsilon(\alpha)}/{\sqrt{N}} )^{-1},\quad\beta\in \mathbb{ H}^+(M_n(\mathbb{C})).$$ 
We have 
$$
B_N^{(1)}(\mathbb{E}{R}_N^g(\beta)-\mathbb{E}R_N^\varepsilon(\beta)) B_N^{(2)}
=\int_0^{\pi/2}\mathbb{E}\left(B_N^{(1)}\frac{\partial}{\partial\alpha}R_N^{\varepsilon,\alpha}(\beta) 
B_N^{(2)}\right)\,\mathrm{d}\alpha,
$$

$$
\frac{\partial}{\partial\alpha}R_N^{\varepsilon,\alpha}(\beta) 
=R_N^{\varepsilon,\alpha}(\beta)\gamma_2\otimes 
\left(\cos\alpha\frac{X^g_N}{\sqrt{N}}
-\sin\alpha\frac{X_N(\varepsilon)}{\sqrt{N}}
\right)
R_N^{\varepsilon,\alpha}(\beta).
$$

Define a basis of the real vector space of  selfadjoint matrices in $M_N(\mathbb{C})$ as follows:
$$\widetilde e_{j,j}=\hat e_{j,j} , 1\leq j \leq N,$$
$$ \widetilde e_{j,k} =:\frac{1}{\sqrt{2}} (\hat e_{j,k} +\hat e_{k,j}), 1 \leq j<k\leq N,$$
$$\widetilde f_{j,k} =:\frac{i}{\sqrt{2}} (\hat e_{j,k} -\hat e_{k,j}), 1 \leq j<k\leq N.$$
In the following calculation, we write simply
$R_N^{\varepsilon,R}$ in place of $R_N^{\varepsilon,\alpha}(\beta)$ and $X^g$ in place of $X^g_N$:
\begin{align*} 
\frac{\partial}{\partial \alpha}R_N^{\varepsilon,\alpha}=&\frac{1}{\sqrt{N}}\sum_{k=1}^N
(-\sin\alpha{X}_{kk}(\varepsilon)+\cos\alpha X^g_{kk})R_N^{\varepsilon,\alpha} 
\beta_2\otimes \widetilde e_{k,k}R_N^{\varepsilon,\alpha}\\
&+\frac{1}{\sqrt{N}}\sum_{1\leq j<k\le N} 
(-\sin\alpha\sqrt{2}\Re{X}_{jk}(\varepsilon)+\cos\alpha\sqrt{2}\Re X^g_{jk}) 
R_N^{\varepsilon,\alpha}
\gamma_2\otimes \widetilde e_{j,k}R_N^{\varepsilon,\alpha}\\
&+\frac{1}{\sqrt{N}}\sum_{1\leq j<k\le N}
(-\sin\alpha\sqrt{2}\Im{X}_{jk}(\varepsilon) +\cos \alpha\sqrt{2} \Im X^g_{jk})R_N^{\varepsilon,\alpha}
\gamma_2\otimes
\widetilde{f}_{j,k}R_N^{\varepsilon,\alpha}.
\end{align*}
Next, we apply Lemma \ref{IPP} with $p=3$ for $1\leq k\leq N$, $ j<k$, to each random variable $Z$  in the set  
$$\{ \sqrt{2}\Re X_{jk}(\varepsilon),\sqrt{2}\Re X^g_{jk},\sqrt{2}\Im  X_{jk}(\varepsilon), 
\sqrt{2}\Im X^g_{jk},X_{kk}(\varepsilon), X^g_{kk}, j<k\}$$ and to each $\phi$ in the set 
\begin{align*}
\{&\mathrm{Tr}(B_N^{(1)}
  R_N^{\varepsilon,\alpha}\gamma_2\otimes \widetilde e_{k,k}R_N^{\varepsilon,\alpha} B_N^{(2)} e_{q,l}\otimes \hat e_{u,v}), \mathrm{Tr}( B_N^{(1)} R_N^{\varepsilon,\alpha} \gamma_2 \otimes \widetilde e_{j,k}R_N^{\varepsilon,\alpha} B_N^{(2)}e_{q,l}\otimes \hat e_{u,v}),\\
& \mathrm{Tr}( B_N^{(1)} R_N^{\varepsilon,\alpha}\gamma_2\otimes \widetilde f_{j,k}R_N^{\varepsilon,\alpha} B_N^{(2)}e_{q,l}\otimes \hat e_{u,v})\colon 1\leq u,v\leq N, 1\leq q,l\leq m \}.
\end{align*}
Setting now $B=B_N^{(2)}e_{q,l}\otimes\hat e_{u,v}B_N^{(1)},$ we have:
\begin{align*}
&\mathrm{Tr} \big(\frac{ \partial }{\partial \alpha} R_N^{\varepsilon,\alpha} B\big)
=\frac{C(\alpha) }{N^\frac32}\sum_{k=1}^N \kappa_3^{k,k,\varepsilon}\mathrm{Tr} 
(R_N^{\varepsilon,\alpha}\gamma_2\otimes\widetilde e_{k,k}R_N^{\varepsilon,\alpha}\gamma_2\otimes 
\widetilde e_{k,k}R_N^{\varepsilon,\alpha}\gamma_2\otimes \widetilde e_{k,k}R_N^{\varepsilon,\alpha} B)\\
&+
\frac{C(\alpha)}{N^\frac32}    \sum_{1\leq j < k<N} \kappa_3^{j,k,\varepsilon}\mathrm{Tr}
(R_N^{\varepsilon,\alpha}\gamma_2\otimes\widetilde e_{j,k}R_N^{\varepsilon,\alpha}\gamma_2\otimes 
\widetilde e_{j,k}R_N^{\varepsilon,\alpha}
 \gamma_2\otimes \widetilde e_{j,k}R_N^{\varepsilon,\alpha} B)\\
&+\frac{C(\alpha) }{N^\frac32} \sum_{1\leq j < k<N} \widetilde\kappa_3^{j,k,\varepsilon}\mathrm{Tr} 
(R_N^{\varepsilon,\alpha}\gamma_2\otimes\widetilde f_{j,k}R_N^{\varepsilon,\alpha}\gamma_2\otimes 
\widetilde f_{j,k}R_N^{\varepsilon,\alpha}
\gamma_2 \otimes \widetilde f_{j,k}R_N^{\varepsilon,\alpha} B)\\
&+\frac{\widetilde C(\alpha)}{  N^2}\sum_{k=1}^N \kappa_4^{k,k,\varepsilon} \mathrm{Tr}
(R_N^{\varepsilon,\alpha}\gamma_2\otimes\widetilde e_{k,k}R_N^{\varepsilon,\alpha}\gamma_2\otimes 
\widetilde e_{k,k}R_N^{\varepsilon,\alpha}
\gamma_2\otimes\widetilde e_{k,k}R_N^{\varepsilon,\alpha}\gamma_2\otimes 
\widetilde e_{k,k}R_N^{\varepsilon,\alpha} B)\\
& +
\frac{\widetilde C(\alpha)}{N^2}\sum_{1\leq j < k<N}\kappa_4^{j,k,\varepsilon} \mathrm{Tr}
(R_N^{\varepsilon,\alpha}\gamma_2\otimes\widetilde e_{j,k}R_N^{\varepsilon,\alpha}\gamma_2\otimes 
\widetilde e_{j,k}R_N^{\varepsilon,\alpha}
\gamma_2\otimes \widetilde e_{j,k}R_N^{\varepsilon,\alpha}\gamma_2\otimes 
\widetilde e_{j,k}R_N^{\varepsilon,\alpha} B)\\
& +\frac{\widetilde C(\alpha)}{ N^2}\sum_{1\leq j < k<N}\widetilde \kappa_4^{j,k,\varepsilon}\mathrm{Tr }
(R_N^{\varepsilon,\alpha}\gamma_2\otimes\widetilde f_{j,k}R_N^{\varepsilon,\alpha}\gamma_2\otimes 
\widetilde f_{j,k}R_N^{\varepsilon,\alpha}\gamma_2\otimes \widetilde f_{j,k}R_N^{\varepsilon,\alpha}\gamma_2\otimes 
\widetilde f_{j,k}R_N^{\varepsilon,\alpha})+\delta\\
&= I_1+I_2+I_3+I_4 + I_5 + I_6 +\delta,
\end{align*}
where 
$$
\left| \delta \right| \leq C_\varepsilon \frac{\left\| (\Im w)^{-1} \right\|^6}{\sqrt{N}},
$$
for some $C_\varepsilon\ge 0$, while $C(\alpha)$ and $\widetilde C(\alpha)$ are polynomials in $\cos \alpha $ and $\sin \alpha$. In the following, $C_\varepsilon$ 
may vary from line to line. It is clear that 
$$
|I_1|\leq C_\varepsilon\frac{\left\|(\Im \beta)^{-1}\right\|^4}{\sqrt{N}},
\text{ and } 
|I_4|\leq C_\varepsilon \frac{\left\| (\Im w)^{-1} \right\|^5}{{N}} .
$$
Next, $I_2 $ and $I_3$ are   a finite linear combinations  of terms of the form 
\begin{equation}\label{termes} 
\frac{C(\alpha) }{N^\frac32}\sum_{j,k \in {\mathcal E}}  C^{j,k,\varepsilon} \mathrm{Tr}_n 
(\gamma_2{(R_N^{\varepsilon,\alpha})}^{p_1p_2}\gamma_2{(R_N^{\varepsilon,\alpha})}^{p_3p_4}\gamma_2({R_N^{\varepsilon,\alpha}} B_N^{(2)})^{p_5 u} e_{q,l}( B_N^{(1)}R_N^{\varepsilon,\alpha})^{v p_6}),
\end{equation}
where ${\mathcal E}$ is some subset of $\{1,\ldots,N\}^2$, $C^{j,k,\varepsilon}\in\{\kappa_3^{j,k,\varepsilon}, 
\widetilde \kappa_3^{j,k,\varepsilon}\}$, and $(p_1, \ldots,p_6)$ is a permutation of $(k,k,k,j,j,j)$. 
The two following cases hold:
\begin{itemize}

\item   $p_5=p_6$, in which case Lemma \ref{prelim} yields
 \begin{align*}   
\|\sum_{j,k \in {\mathcal E}}&  C^{j,k,\varepsilon} \mathrm{Tr}_n 
(\gamma_2{(R_N^{\varepsilon,\alpha})}^{p_1p_2}
\gamma_2{(R_N^{\varepsilon,\alpha})}^{p_3p_4}
\gamma_2({R_N^{\varepsilon,\alpha}} B_N^{(2)})^{p_5 u} e_{q,l}( B_N^{(1)}R_N^{\varepsilon,\alpha})^{v p_6})\|\\
&\leq  C_\varepsilon\|\gamma_2\|^3\|(\Im \beta)^{-1}\|^2N 
\sum_{j=1}^N \|(B_N^{(1)}{R_N^{\varepsilon,\alpha}})^{vj}\|\|({R_N^{\varepsilon,\alpha}} B_N^{(2)})^{ju}\|\\
&\leq C_\varepsilon\|\gamma_2\|^3 \|(\Im \beta)^{-1}\|^2N 
\Big(\sum_{j=1}^N \|(B_N^{(1)}{R_N^{\varepsilon,\alpha}})^{vj}\|^2\Big)^{1/2} 
\Big(\sum_{j=1}^N\| ({R_N^{\varepsilon,\alpha}} B_N^{(2)})^{ju}\|^2 \Big)^{1/2}\\ 
&\leq C_\varepsilon\left\|\gamma_2 \right\|^3 \|(\Im \beta )^{-1} \|^4nN.
\end{align*}
\item $p_5\ne p_6$, in which case Lemma \ref{prelim} yields 
\begin{align*} 
\|\sum_{j,k \in {\mathcal E}}&  C^{j,k,\varepsilon} \mathrm{Tr}_n 
(\gamma_2{(R_N^{\varepsilon,\alpha})}^{p_1p_2}
\gamma_2{(R_N^{\varepsilon,\alpha})}^{p_3p_4}
\gamma_2({R_N^{\varepsilon,\alpha}} B_N^{(2)})^{p_5 u} e_{q,l}( B_N^{(1)}R_N^{\varepsilon,\alpha})^{v p_6})\|\\
&\leq  C_\varepsilon\|\gamma_2\|^3 \|(\Im \beta)^{-1}\|^2
\Big(\sum_{k=1}^N 
\|({R_N^{\varepsilon,\alpha}} B_N^{(2)})^{ku}\|\Big)\Big(\sum_{j=1}^N\|( B_N^{(1)}{R_N^{\varepsilon,\alpha}})^{vj}\|\Big) \\
&\leq  C_\varepsilon N \|\gamma_2\|^3 \|(\Im \beta)^{-1} \|^2
\Big(\sum_{k=1}^N 
\|({R_N^{\varepsilon,\alpha}} B_N^{(2)})^{ku}\|^2\Big)^{1/2}\Big(\sum_{j=1}^N\|(B_N^{(1)}{R_N^{\varepsilon,\alpha}})^{vj}\|^2
\Big)^{1/2}\\
&\leq  C_\varepsilon n N \left\|\gamma_2\right\|^3 \left\|(\Im \beta)^{-1} \right\|^4.
\end{align*}
\end{itemize}
We see now that  
$
| I_j | \leq C_\varepsilon {\| (\Im \beta)^{-1} \|^4}/{\sqrt{N}}$ for $j=2,3$.
Finally, $I_5 $ and $I_6$ are finite linear combinations  of terms of the form 
\begin{equation}\label{termes2} 
\frac{\widetilde C(\alpha)}{N^2}\sum_{j,k \in {\mathcal E}} C^{j,k,\varepsilon}\mathrm{Tr}_n 
( \gamma_2{(R_N^{\varepsilon,\alpha})}^{p_1p_2}\gamma_2 {(R_N^{\varepsilon,\alpha})}^{p_3p_4}\gamma_2 {(R_N^{\varepsilon,\alpha})}^{p_5p_6}\gamma_2({R_N^{\varepsilon,\alpha}}B_N^{(2)})^{p_7 u} e_{q,l} (B_N^{(1)}{R_N^{\varepsilon,\alpha}})^{v p_8})
\end{equation}
where ${\mathcal E}$ is some subset of $\{1,\ldots,N\}^2$, 
$C^{j,k,\varepsilon}\in\{\kappa_4^{j,k,\varepsilon},\widetilde\kappa_4^{j,k,\varepsilon}\}$ and $(p_1, \ldots,p_6)$ is a permutation of $(k,k,k,k,j,j,j,j)$. Lemma \ref{prelim} shows that
the norm of such a term can be estimated by
\begin{align*}
 &\frac{1}{N^2}C_\varepsilon N \|\gamma_2 \|^4 
\|(\Im \beta)^{-1} \|^4
\Big(\sum_{k=1}^N \| ({R_N^{\varepsilon,\alpha}}B_N^{(2)})^{ku}\|\Big) \\
&\le \frac{1}{N^2} C_\varepsilon  N^\frac32 \|\gamma_2  \|^4 \|(\Im \beta)^{-1} \|^4
\Big(\sum_{k=1}^N \|({R_N^{\varepsilon,\alpha}} B_N^{(2)})^{ku}\|^2\Big)^{1/2}
\le \frac{1}{N^2} C_\varepsilon \sqrt{n} N^\frac32\|\gamma_2\|^4 \|(\Im \beta)^{-1} \|^5.
 \end{align*}
It follows that $
\left| I_j \right| \leq C_\varepsilon  {\left\| (\Im w)^{-1} \right\|^5}/{\sqrt{N}} 
$ for $j=5,6 $.
The proposition follows.
\end{proof}

We show next that $\mathcal{F}_N^\varepsilon(\beta)$ is close to its expected value. This result uses concentration inequalities in the presence of a Poincar\'e inequality. 
We recall that if the law of a random variable $X$ satisfies the Poincar\'e inequality with 
constant $C$ and  $\alpha\in\mathbb R\setminus\{0\}$, then the law of $\alpha X$ satisfies the Poincar\'e 
inequality with constant $\alpha^2 C$. Moreover, suppose that the probability measures 
$\mu_1,\ldots,\mu_r$ on $\mathbb{R}$ satisfy the Poincar\'e inequality with constants 
$C_1,\ldots,C_r$ respectively. Then the product measure 
$\mu_1\otimes \cdots \otimes \mu_r$ on $\mathbb{R}^r$ satisfies the Poincar\'e 
inequality with constant $C=\max\{C_1,\dots,C_r\}$. That is, if
$f:\mathbb C^r\to\mathbb R$ is an arbitrary differentiable function  such that $f$ and its gradient ${\rm grad} f$ are square integrable relative to $\mu_1\otimes \cdots \otimes \mu_r$, then
$$
\mathbb{V}(f)\leq C \int_{\mathbb R^r} \Vert {\rm grad} f \Vert_2 ^2 d\mu_1\otimes \cdots \otimes \mu_r.
$$
Here $\mathbb{V}(f)=\int|f-\int f\,{\rm d}\mu_1\otimes\cdots\otimes\mu_r|^2\,{\rm d}
\mu_1\otimes\cdots\otimes\mu_r$ (see \cite[Theorem 2.5]{GuZe03}).

We use the following concentration result (see {\cite[Lemma 4.4.3 and Exercise 4.4.5]{AGZ} or \cite[Chapter 3]{L}}).

\begin{lemma}\label{Herbst}
	Let $\mathbb{P}$ be a probability measure on $\mathbb{R}^r$ which satisfies a Poincar\'e inequality 
	with constant $C$. Then there exist $K_1>0$ and $K_2>0$ such that,  for every  Lipschitz function 
	$F$ on $\mathbb{R}^r$ with Lipschitz constant $|F|_{\text{\rm Lip}}$, and for every $\varepsilon>0$, we have
	$$
	 \mathbb{P}\left(\vert F-\mathbb E_{\mathbb{P}}(F)\vert>\varepsilon\right)\leq 
	K_1 \exp\Big(-\frac{\varepsilon}{K_2\sqrt{C}|F|_{\text{\rm Lip}}}\Big).$$
\end{lemma}

The following result is similar to Proposition \ref{prop.6.2}(1).

\begin{prop}
	Suppose that $T_N,S_N\in M_N(\mathbb C)$ are contractions of uniformly bounded rank.  Given $\varepsilon>0$ and  $\beta\in \mathbb H^+(M_n(\mathbb C))$, almost surely
	\begin{equation*}
	\lim_{N\to\infty}\|(I_n\otimes T_N)(R_N^\varepsilon(\beta)-
	\mathbb E
	R_N^\varepsilon(\beta))(
	I_n\otimes S_N)\|=0.
	\end{equation*}
\end{prop}
\begin{proof}
	As in the proof of  Proposition \ref{prop.6.2}, it suffices to consider the case in which $T_N$ and $S_N$ are contractions of rank $1$. Write $I_n$ as a sum $Q_1+\cdots+Q_n$ of rank $1$ projections. Then the norm in the statement
	is at most equal to
	\begin{equation*}
	\sum_{j,k=1}^n \|(Q_j\otimes T_N)(R_N^\varepsilon(\beta)-
	\mathbb E
	R_N^\varepsilon(\beta))
	(Q_k\otimes S_N)\|
	=\sum_{j,k=1}^n |\mathrm{Tr}(R_N^\varepsilon(\beta)-
	\mathbb E
	R_N^\varepsilon(\beta))Q_{j,k}|,
	\end{equation*}
	where each $Q_{j,k}$ is a contraction of rank $1$. Given a selfadjoint matrix $Z_N
	\in M_N(\mathbb C)$, we set
	$
	R(Z,\beta)=
	(\beta\otimes I_N-L(C_N,Z))^{-1}$
	and $f_{N,j,k}(Z)=\mathrm{Tr} R(Z,\beta)Q_{j,k}$. We have
	\begin{equation*}
	f_N(Z_1)-f_N(Z_2)=\mathrm{Tr}(R(Z_1,\beta)(\gamma_2\otimes(Z_1-Z_2))R(Z_2,\beta)Q_{j,k}), 
	\end{equation*}
	and thus
	\begin{align*}
	|f_N(Z_1)-f_N(Z_2)|&\le \|\gamma_2\otimes(Z_1-Z_2)\|_2
	\|R(Z_1,\beta)Q_{j,k}R(Z_2,\beta)\|_2\\
	&\le \|\gamma_2\|_2\|(Z_1-Z_2)\|_2\|(\Im\beta)^{-1}\|^2.
	\end{align*}
	An application of Lemma \ref{Herbst} and of the comment preceding it yield
	\begin{equation*}
	\mathbb P(|\mathrm{Tr}(R_N^\varepsilon(\beta)-
	\mathbb E
	R_N^\varepsilon(\beta))Q_{j,k}|>\delta)
	\le 2\exp(-CN^{1/2}\|(\Im\beta)^{-1}\|^{-2}
	\delta)
	\end{equation*}
	for every $\delta>0$, with a constant $C$ that does not depend in $N,j,$ or $k$. The proposition follows by an application of the Borel-Cantelli lemma.	
\end{proof}

\begin{cor}
For every $\beta \in \mathbb{H}^+(M_n(\mathbb{C}))$ and every $\varepsilon>0$ we have, almost surely,
\begin{equation}\label{concentre}
\lim_{N\to\infty}\mathcal{F}_N^{\varepsilon}(\beta)-\mathbb{E}(\mathcal{F}_N^{\varepsilon}(\beta))=0.
\end{equation}
\end{cor}

Observe now that
\begin{align*}
\mathcal{F}_N (\beta)-&(\omega (\beta) -s\gamma_1)^{-1}\otimes I_p
=  \mathcal{F}_N(\beta)-\mathcal{F}_N^\varepsilon(\beta)+\mathcal{F}_N^\varepsilon(\beta)-\mathbb{E}(\mathcal{F}_N^\varepsilon(\beta))\\
&+
\mathbb{E}(\mathcal{F}_N^\varepsilon  
(\beta))-\mathbb{E}( \mathcal{F}^g_N(\beta))+\mathbb{E}(\mathcal{F}^g_N(\beta))  -(\omega (\beta) -s\gamma_1)^{-1}\otimes I_p.
\end{align*}
We let $N\to\infty$  and then $\varepsilon\to0$ and apply  \eqref{di}, \eqref{concentre}, 
Proposition \ref{comptilde}, and Proposition \ref{gaussien} to obtain the following result.

\begin{thm}\label{thwigner}
For every $\beta\in \mathbb{H}^+(M_n(\mathbb{C}))$ we have, almost surely, when $N$ goes to infinity, $\lim_{N\to\infty}\mathcal{F}_N(\beta)=(\omega (\beta) -s\gamma_1)^{-1}\otimes I_p$.
\end{thm}
Everything is now in place for completing the argument.

\begin{proof}[Proof of Theorem {\em\ref{thm8.2}}]
We noted earlier that $\omega(ze_{1,1}-\gamma_0)$ is analytic in 
 $\mathbb C\setminus\sigma(P(a,b))$.
For fixed $t\in\mathbb R\setminus\sigma(P(a,b))$, set
$\Psi(\beta)=\beta \otimes 1 -L(a,b)$
and $\Psi_N(\beta)=\beta \otimes 1 -L(C_N,X_N/\sqrt{N})$.  
According to Lemma \ref{o lema}, $\Psi(te_{1,1}-\gamma_0)$  is invertible, and thus there exists $\delta>0$ such that  
\begin{equation}\label{dist}
d(0, \sigma(\Psi(te_{1,1}-\gamma_0)))\geq \delta>0.\end{equation}
  Theorem \ref{stronguni} and Proposition \ref{Camille} imply that, almost surely, for every complex polynomial $Q$ in one variable we have  
$$
\lim_{N\to\infty}\Vert Q(\Psi_N(te_{1,1}-\gamma_0)\Vert=\Vert Q( \Psi(te_{1,1}-\gamma_0))\Vert.
$$
Asymptotic freeness implies that almost surely,  for every complex polynomial $Q$ in one variable,
$$
\lim_{N\to\infty}(\mathrm{tr}_n\otimes\mathrm{tr}_N)( Q(\Psi_N(te_{1,1}-\gamma_0)))=(\mathrm{tr}_n\otimes\tau)( Q(\Psi(te_{1,1}-\gamma_0))).
$$
Thus,
denoting the Hausdorff distance by $d_H$,
we deduce that almost surely for $N $ large enough,  
$$
d_H(\sigma(\Psi_N(te_{1,1}-\gamma_0)), \sigma(\Psi(te_{1,1}-\gamma_0))) \leq \delta/4.$$
Note that $\Psi_N(te_{1,1}-\gamma_0)$ is selfadjoint. For an arbitrary $\beta \in M_n(\mathbb{C})$, we have
\begin{equation*}
\|\Psi_N(te_{1,1}-\gamma_0)-\Psi_N(\beta)\|
\le\|(te_{1,1}-\gamma_0)-\beta\|.
\end{equation*}
 It follows from \eqref{dist} that, almost surely for all large $N$, if $\|\beta-(te_{1,1}-\gamma_0)\|<\delta/4$, then 
$$
d(0, \sigma(\Psi_N(\beta))) \geq \delta/2.
$$
Moreover, denoting by $s_1(M)$ the smallest singular value of an arbitrary matrix $M$, 
$$
s_1(\Psi_N(\beta))\geq s_1(\Psi_N(te_{1,1}-\gamma_0))-\Vert \beta-(te_{1,1}-\gamma_0)\Vert.
$$ 
Thus,  almost surely for all large $N$, provided that $\|\beta-(te_{1,1}-\gamma_0)\|<\delta/4$, we have
$
\|\mathcal{F}_N(\beta)\|=\Vert (\Psi_N(\beta))^{-1}\Vert \leq  4/\delta.
$
In other words, almost surely, the family $\{\mathcal{F}_N\}_{N\in\mathbb N}$ is normal
in a neighborhood of $te_{1,1}-\gamma_0$. 
According to Theorem \ref{thwigner}, for any $\beta\in\mathbb H^+
(M_n(\mathbb C))$, almost surely  $\mathcal{F}_N$ converges towards $(\omega( \beta)-s \gamma_1 )^{-1}\otimes I_p$.
Set 
$$
\Lambda= \{w\in M_n(\mathbb{C}), \Vert w-(te_{1,1}-\gamma_0)\Vert<\delta/4,\; \Im w >0\}.
$$
Almost surely for any $w\in\Lambda$ such that $\Im w\in M_n(\mathbb{Q})$ and $\Re w\in 
M_n(\mathbb{Q})$, $\mathcal{F}_N (w)$ converges towards $(\omega( w)-s \gamma_1 )^{-1}\otimes I_p$.
The Vitali-Montel convergence theorem implies that that almost surely $\mathcal{F}_N$ converges towards a holomorphic function on 
$\{w\in M_n(\mathbb C), \|w-(te_{1,1}-\gamma_0)\|<\delta/4\},$ and, in particular, 
$\mathcal{F}_N (ze_{1,1}-\beta_0)$ converges for any $z\in\mathbb C$ such that $|z-t|<\delta/4$ towards
$(\omega ((ze_{1,1}-\gamma_0) -s\gamma_1)^{-1}\otimes I_p$.

 Now, the Hurwitz theorem on zeros of analytic functions implies that, almost surely for large $N$, the function $F_N(ze_{1,1}-\gamma_0)
=\det(I_n\otimes I_p-\mathcal{F}_N(ze_{1,1}-\gamma_0))
$ has as many zeros in a neighborhood of $t$ as the function
\begin{equation*}
\det(I_n\otimes I_p-(\gamma_1\otimes T)((\omega ((ze_{1,1}-\gamma_0) -s\gamma_1)^{-1}\otimes I_p).
\end{equation*}
Now, note that 
\begin{align*}
\gamma_1((\omega(ze_{11}-\gamma_0)&-s\gamma_1)^{-1}
\otimes T-I_n\otimes I_p \\
& =  \left(\gamma_1\otimes T
-s\gamma_1\otimes I_p\right)
 (\omega(ze_{11}-\gamma_0)-s\gamma_1)^{-1}\otimes I_p-I_n\otimes I_p\\
&= \left(\gamma_1\otimes T-\omega(ze_{11}-\gamma_0)\otimes I_p
\right)(\omega(ze_{11}-\gamma_0)-s\gamma_1)^{-1}\otimes I_p
\end{align*}
and 
$$ 
\omega(ze_{1,1}-\gamma_0) \otimes I_p -\gamma_1 \otimes T
=\sum_{i=1}^p (\omega(ze_{11}-\gamma_0) -\gamma_1 \theta_i)\otimes e_{i,i}.
$$
Therefore
$$
\det( \omega(ze_{11}-\gamma_0) \otimes I_p -\gamma_1 \otimes T)
=\prod_{i=1}^p \det (\omega(ze_{11}-\gamma_0) -\gamma_1\theta_i).
$$
The theorem follows. 
\end{proof}

\section{Appendix}
The following result is \cite[Lemma 8.1]{BC}.

\begin{lemma}\label{prelim} For any matrix $M \in M_n(\mathbb{C})\otimes M_N(\mathbb{C})$, 

 \begin{equation}\label{lp}
 \frac{1}{N} \sum_{k,l =1}^N \|M^{kl} \|^2 \leq n \|M \|^2
 \end{equation}
and for any fixed $k$, \begin{equation}\label{O}  \sum_{l =1}^N \|M^{lk} \|^2 \leq n \|M \|^2\end{equation}
and \begin{equation}\label{l} \sum_{l =1}^N \|M^{kl} \|^2 \leq n \|M \|^2,\end{equation}
where $M^{kl}$ is defined by \eqref{mij}.
\end{lemma}



\begin{thebibliography}{10}

\bibitem{A} Lars V. Ahlfors. {\em
An Introduction to the Theory of Analytic Functions of One Complex Variable.}
Third Edition (1979) McGraw-Hill, Inc. New York.

\bibitem{akhieser} Naum Ilich Akhieser, {\em The classical moment problem
and some related questions in analysis.} Translated by N. Kemmer.
Hafner Publishing Co., New York, 1965.

\bibitem{AGZ} G. W. Anderson, A. Guionnet, and O. Zeitouni, \emph{An
introduction to random matrices}, Cambridge University Press, Cambridge,
2010.

\bibitem{Anderson} G. W. Anderson, \textit{Convergence of the largest singular value of a polynomial 
in independent Wigner matrices}. Ann. Probab. {\bf 41} (2013), 2103--2181.

\bibitem{BaiYao08b} Z. D. Bai and J. Yao, On sample eigenvalues in
a generalized spiked population model, \emph{J. Multivariate Anal}.
\textbf{106} (2012), 167--177. 

\bibitem{BYi} Z. D. Bai and Y. Q. Yin, {\em Necessary and sufficient conditions for almost sure convergence
of the largest eigenvalue of a Wigner matrix}. Ann. Probab., 16(4), (1988), 1729--1741.

\bibitem{BBP05} J. Baik, G. Ben Arous, and S. P{\'e}ch{\'e}, Phase
transition of the largest eigenvalue for nonnull complex sample covariance
matrices, {\em Ann. Probab.} \textbf{33} (2005), 1643--1697.

\bibitem{BaikSil06} J. Baik and J. W. Silverstein, Eigenvalues of
large sample covariance matrices of spiked population models, \emph{J.
Multivariate Anal}. \textbf{97}(2006), 1382--1408.

\bibitem{CAOT} S. T. Belinschi, \emph{Some Geometric Properties of the Subordination Function 
Associated to an Operator-Valued Free Convolution Semigroup} Complex Anal. Oper. Theory (2017), 
DOI 10.1007/s11785-017-0688-y

\bibitem{BPV}
S. T. Belinschi, M. Popa, and V. Vinnikov, 
\emph{Infinite divisibility and a 
non-commutative {B}oolean-to-free {B}ercovici-{P}ata 
bijection}, Journal of Functional Analysis, {\bf 262}, 2012, Issue 1, 94--123.

\bibitem{BBCF} S. T. Belinschi, H. Bercovici, M. Capitaine and M. F\'evrier, {\em Outliers in 
the spectrum of large deformed unitarily invariant models}. Ann.
Probab., 45 (2017), No. 6A, 3571--3625.

\bibitem{BC} S. T.  Belinschi and M. Capitaine, {\em Spectral properties of polynomials in independent 
Wigner and deterministic matrices}. Journal of Functional Analysis 273 (2017) 3901--3963.

\bibitem{BMS} S. T. Belinschi, T. Mai and R. Speicher, {\em Analytic 
subordination theory of operator-valued free additive convolution 
and the solution of a general random matrix problem.} 
Journal f\"{u}r die reine und angewandte Mathematik, 2015.

\bibitem {BGR} F. Benaych-Georges and R. R. Nadakuditi,
 {\em The eigenvalues and eigenvectors of finite, low rank 
perturbations of large random matrices.} 
Advances in Mathematics 227 (2011), 494--521.

\bibitem{BVReg} {H. Bercovici and D. Voiculescu,} {\em Regularity questions for free convolution.} 
{Nonselfadjoint operator algebras, operator theory, and related topics,} {37--47}, {Oper. Theory 
Adv. Appl.} {104}, Birkh\"{a}user, Basel, 1998.

\bibitem{biane} Ph. Biane, {\em On the free convolution with a semi-circular distribution}. 
Indiana Univ Math. J. {\bf46}(3), 705--718 (1997).

\bibitem{Bruce} B. Blackadar, Operator Algebras.
 Theory of $C^*$-Algebras and von Neumann Algebras. 
Encyclopaedia of Mathematical Sciences, Volume 122. 
Springer-Verlag Berlin Heidelberg 2006.

\bibitem{C} M. Capitaine, \emph{Additive/multiplicative free
subordination property and limiting eigenvectors of spiked additive
deformations of Wigner matrices and spiked sample covariance matrices,}
Journal of Theoretical Probability, Volume 26 (3) (2013), 595--648.

\bibitem{Capitaine14} M. Capitaine, \emph{
Exact separation phenomenon for the eigenvalues 
of large Information-Plus-Noise type matrices. 
Application to spiked models,}
{Indiana Univ. Math. J.}  {\bf 63} (2014),  1875--1910.

\bibitem{CD} M. Capitaine and C. Donati-Martin. {\em Strong asymptotic freeness for Wigner and
Wishart matrices.} Indiana Univ. Math. J., {\bf56} (2):767-–803, 2007.

\bibitem{CDF09} M. Capitaine, C. Donati-Martin, and D. F{\'e}ral,
\emph{The largest eigenvalues of finite rank deformation of large Wigner
matrices: convergence and nonuniversality of the fluctuations}, Ann.
Probab.\textbf{ 37} (2009), 1--47.

\bibitem{CDMFF} M. Capitaine, C. Donati-Martin, D. F\'eral and M. F\'evrier. 
{\em  Free convolution with a semicircular distribution and 
eigenvalues of spiked deformations of Wigner matrices.} Electronic
Journal of Probability 16 (2011), 1750--1792.

\bibitem{CM} B. Collins and C. Male, \emph{The strong asymptotic
freeness of Haar and deterministic matrices}, 
 Ann. Sci. \'Ec. Norm. Sup\'er. (4)  {\bf 47} (2014), no. 1, 147--163.

\bibitem{CHS} B. Collins, T. Hasebe and N. Sakuma, 
\emph{Free probability for purely discrete
eigenvalues of random matrices}. J. Math. Soc. Japan 70, No.3 (2018) 1111--1150.

\bibitem{D} K.J. Dykema, 
\emph{On certain free product factors via an extended matrix model.} 
J. Funct. Anal. {\bf 112} (1993), 31--60.

\bibitem{FePe} D. F{\'e}ral and S. P{\'e}ch{\'e}, \emph{The largest eigenvalue
of rank one deformation of large Wigner matrices}, Comm. Math.
Phys. \textbf{272} (2007), 185--228.

\bibitem{FurKom81} Z. F\H{u}redi and J. Koml{\'o}s, \emph{The eigenvalues
of random symmetric matrices}, {Combinatorica} \textbf{1} (1981),
233--241.

\bibitem{GuZe03}
A.~Guionnet and B.~Zegarlinski.
\newblock Lectures on Logarithmic Sobolev inequalities.
\newblock In {\em S\'eminaire de {P}robabilit\'es, {XXXVI}}, volume 1801 of
  {\em Lecture Notes in Math.}. Springer, Berlin, 2003.

\bibitem{HT} U. Haagerup and S. Thorbj{\o}rnsen,
 {\em A new application of random matrices:
${\rm Ext}({\rm C}^*_{\rm red}(F_2))$ is not a group}, 
Ann. of Math. (2) {\bf 162} (2005), no. 2, 711--775.

\bibitem{HRS} W. Helton, R. Rashidi-Far and R. Speicher,
{\em Operator-valued Semicircular Elements: Solving A Quadratic Matrix Equation with Positivity 
Constraints}. Int. Math. Res. Not. 2007, No. 22, Article ID rnm086, (2007)

\bibitem{John} I. Johnstone, \emph{On the distribution of the largest eigenvalue
in principal components analysis}, Ann. Stat. \textbf{29} (2001),
295--327.

\bibitem{L} M. Ledoux. 
The Concentration of Measure Phenomenon, Mathematical Surveys and Monographs,
Volume 89, A.M.S, 2001.

\bibitem{LV}P. Loubaton and P. Vallet,  {\em Almost sure localization of
the eigenvalues in a Gaussian information-plus-noise model. Application
to the spiked models}, Electron. J. Probab. {\bf 16} (2011), no. 70, 1934--1959.

\bibitem{Mai} T. Mai, {On the Analytic Theory of Non-commutative Distributions in Free Probability,} PhD thesis, Universität des Saarlandes, 2017, http://scidok.sulb.uni-saarland.de/volltexte/2017/6809.

\bibitem{My} M. Ma{i}da, 
\emph{ Large deviations for the largest eigenvalue of rank one deformations of Gaussian ensembles,}
Elec. J. Probab. Vol. {\bf 12} (2007) 1131--1150.

\bibitem{M} C. Male, {\em The norm of polynomials in large random and
deterministic matrices.} With an appendix by Dimitri Shlyakhtenko.
{Probab. Theory Related Fields} {\bf 154}, no. 3-4, 477–-532 (2012). 

\bibitem{NSS} A. Nica, D. Shlyakhtenko, and R. Speicher, 
{\em Operator-Valued Distributions. I. Characterizations of Freeness}. 
International Mathematics Research Notices 2002, {\bf 29},
1509--1538.

\bibitem{PS} L. Pastur and Mariya Shcherbina, 
{\em Eigenvalue distribution of Large Random Matrices}
Mathematical Surveys and Monographs, 171. 
American Mathematical Society, Providence, RI, 2011.

\bibitem{Peche06} S. P{\'e}ch{\'e}, \emph{The largest eigenvalue of small
rank perturbations of Hermitian random matrices}, Probab. Theory
Related Fields \textbf{134} (2006), 127--173.

\bibitem{PRS} A. Pizzo, D. Renfrew, and A. Soshnikov, \emph{On finite rank
deformations of Wigner matrices}, 
Ann. Inst. Henri Poincar\'e Probab. Stat., {\bf 49}, (2013), no. 1, 64–94..

\bibitem{RaoSil09} N. R. Rao and J. W. Silverstein, Fundamental limit
of sample generalized eigenvalue based detection of signals in noise
using relatively few signal-bearing and noise-only samples, \emph{IEEE
Journal of Selected Topics in Signal Processing} \textbf{4} (2010),
468--480.

\bibitem{MPS} M. P. Sch\"utzenberger, {\em On the definition of a family of automata},
Inform.Control {\bf 4} (1961), 245--270.

\bibitem{Sh1996}
D. Shlyakhtenko, \emph{Random {G}aussian band matrices and freeness with
  amalgamation}, Internat. Math. Res. Notices (1996), no.~20, 1013--1025.
  \MR{1422374 (97j:46070)}

\bibitem{Sh1999} D. Shlyakhtenko, 
\emph{{$A$}-valued semicircular systems}, J. Funct. Anal. \textbf{166}
  (1999), no.~1, 1--47. \MR{MR1704661 (2000j:46124)}.

\bibitem{ShlB} D. Shlyakhtenko, 
\emph{Free probability of type $B$ and asymptotics of finite rank 
perturbations of random matrices.} Indiana Univ. Math. J. 67 (2018), 971--991.

\bibitem{Takesaki1} Masamichi Takesaki, {\em Theory of
Operator Algebras} I. Springer-Verlag New York Inc., 1979

\bibitem{DVJFA} D. Voiculescu, 
{\em Addition of certain non-commutative random variables}. 
{J. Funct. Anal.}, 66 {(1986)}, 323--346.

\bibitem{V2}  D. Voiculescu, 
{\em Multiplication of certain noncommuting random variables},
J. Operator Theory {\bf 18}(1987), 223--235.

\bibitem{V-Invent}  D. Voiculescu,  
{\em Limit laws for random matrices and free products}.
 { Invent. Math.},{\bf 104} (1991), 201--220.

\bibitem{V3}  D. Voiculescu, 
{\em The analogues of entropy and of Fisher's 
information measure in free probability theory. I}, {Comm. Math. Phys.} 
{\bf 155} {(1993)},  {411--440}.

\bibitem{V1995}  D. Voiculescu, 
\emph{Operations on certain non-commutative operator-valued
  random variables}, Ast\'erisque (1995), no.~232, 243--275, Recent advances in
  operator algebras (Orl{\'e}ans, 1992). \MR{1372537 (97b:46081)}

\bibitem{V2000} D. Voiculescu,  
\emph{The coalgebra of the free difference quotient and free
  probability}, Internat. Math. Res. Notices (2000), no.~2, 79--106.
  \MR{MR1744647 (2001d:46096)}

\bibitem{FreeMarkov}  D. Voiculescu, 
\emph{Analytic subordination consequences of free Markovianity}, 
Indiana Univ. Math. J. {\bf 51} (2002), 1161--1166.

\bibitem{VDN} D. Voiculescu, K. Dykema and A. Nica, {\em Free random
variables.} CRM Monograph Series, Vol. 1, AMS, 1992.

\end{thebibliography}
\end{document}